\documentclass[12pt,a4paper,reqno]{amsart}
\usepackage{color}
\usepackage{amsfonts,amsmath,amssymb,amsxtra,url,float} 
\allowdisplaybreaks[4] 
\usepackage[colorlinks,linkcolor=RoyalBlue,anchorcolor=Periwinkle,citecolor=Orange,urlcolor=Green]{hyperref} 
\usepackage[usenames,dvipsnames]{xcolor} 
\usepackage{enumitem}
\usepackage{geometry} 
\usepackage{rotating} 
\usepackage{lscape} 
\usepackage{graphicx} 
\usepackage{subfigure} 
\usepackage{tikz}
\usetikzlibrary{decorations.pathreplacing,decorations.markings}
 \tikzset{
  on each segment/.style={
    decorate,
    decoration={
      show path construction,
      moveto code={},
      lineto code={
        \path [#1]
        (\tikzinputsegmentfirst) -- (\tikzinputsegmentlast);
      },
      curveto code={
        \path [#1] (\tikzinputsegmentfirst)
        .. controls
        (\tikzinputsegmentsupporta) and (\tikzinputsegmentsupportb)
        ..
        (\tikzinputsegmentlast);
      },
      closepath code={
        \path [#1]
        (\tikzinputsegmentfirst) -- (\tikzinputsegmentlast);
      },
    },
  },
  mid arrow/.style={postaction={decorate,decoration={
        markings,
        mark=at position .5 with {\arrow[#1]{stealth}}
      }}},
}
\usetikzlibrary{arrows}
\usetikzlibrary{trees}

\usetikzlibrary{matrix}
\usetikzlibrary{patterns}
\usetikzlibrary{shadings} 
\usepackage[all]{xy} 
\usepackage{mathdots} 
\usepackage{yhmath} %
\usepackage{dsfont} 
\usepackage{cite}
\usepackage{mathrsfs} 
\usepackage{multicol} 
\numberwithin{figure}{section}

\usepackage{marginnote} 
\usepackage{graphicx} 
\usepackage{multicol} 

\usepackage{changepage} 
\usepackage{changes}  

\newtheorem{theorem}{Theorem}[section]
\newtheorem{lemma}[theorem]{Lemma}
\newtheorem{corollary}[theorem]{Corollary}
\newtheorem{main theorem}[theorem]{Main Theorem}

\newtheorem{definition}[theorem]{Definition}

\newtheorem{remark}[theorem]{Remark}
\newtheorem{example}[theorem]{Example}

\newtheorem{question}[theorem]{Question}

\usetikzlibrary{arrows}

\numberwithin{equation}{section}

\def\<{\langle} 
\def\>{\rangle} 
\def\NN{\mathbb{N}} 

\newcommand{\Pic}{F{\tiny{IGURE}}\ }
\newcommand{\modcat}{\mathsf{mod}}

\newcommand{\Gproj}{\mathsf{G}\text{-}\mathsf{proj}}

\newcommand{\ind}{\mathsf{ind}}
\newcommand{\kk}{\mathds{k}} 
\newcommand{\Q}{\mathcal{Q}} 
\newcommand{\T}{\mathcal{T}}
\newcommand{\I}{\mathcal{I}} 
\newcommand{\J}{\mathcal{J}}
\newcommand{\R}{\mathcal{R}}
\newcommand{\perR}{\mathcal{R}_{\mathrm{p}}}
\newcommand{\Hom}{\mathrm{Hom}} %
\newcommand{\Basis}{\mathfrak{B}} %
\newcommand{\Irr}{\mathrm{Irr}} 
\newcommand{\End}{\mathrm{End}} %
\newcommand{\rad}{\mathrm{rad}} %

\def\defines{\it\color{blue!75}}
\def\s{\mathfrak{s}}
\def\t{\mathfrak{t}}
\def\M{\mathds{M}}
\def\str{\mathrm{Str}}
\def\band{\mathrm{Ban}}
\def\image{\mathrm{Im}}
\def\kernel{\mathrm{Ker}}
\def\spl{\times}
\def\itLamb{\mathit{\Lambda}}
\def\heart{{\color{red}\pmb{\heartsuit}}}
\def\diam{{\color{red}\pmb{\diamondsuit}}}
\newcommand{\Right}{\mathrm{R}}
\newcommand{\Left}{\mathrm{L}}
\newcommand{\C}{\mathscr{C}}
\newcommand{\e}{\varepsilon}
\newcommand{\larrow}[1]{\mathop{\longleftarrow}\limits^{#1}}
\newcommand{\rarrow}[1]{\mathop{\longrightarrow}\limits^{#1}}

\newcommand{\arrmod}{\mathsf{arr}} %
\newcommand{\CMA}{\mathrm{CMA}}

\begin{document}

\title{On endomorphism algebras of string almost gentle algebras}

\author{Yu-Zhe Liu}
\address{School of Mathematics and statistics, Guizhou University, 550025 Guiyang, Guizhou, P. R. China}
\email{liuyz@gzu.edu.cn / yzliu3@163.com}

\author{Panyue Zhou$^*$}
\address{School of Mathematics and Statistics, Changsha University of Science and Technology, 410114 Changsha, Hunan,  P. R. China }
\email{panyuezhou@163.com}

\thanks{$^\ast$Corresponding author.}


\definecolor{section}{rgb}{0,0,0.5}

\maketitle

\vspace{-3mm}

\begin{abstract}
For any arbitrary string almost gentle algebra, we consider specific subsets of its quiver's arrow set, denoted by $\mathcal{R}$. For each such $\mathcal{R}$, we introduce the finitely generated module $M_{\mathcal{R}}$ and define its associated $\mathcal{R}$-endomorphism algebra $A_{\mathcal{R}}$. In this paper, we show that the representation type of a string gentle algebra $A$, the representation type of the $\mathcal{R}$-endomorphism algebra $A_{\mathcal{R}}$ for some $\mathcal{R}$, the representation types of all $\mathcal{R}$-algebras, and the representation type of the Cohen-Macaulay Auslander algebra $A^{\mathrm{CMA}}$ of $A$ are equivalent.
The results presented here reveal a deep structural connection between different classes of algebras derived from string gentle algebras. By showing the equivalence of representation types, this work offers new insights into the nature of endomorphism algebras and Cohen-Macaulay Auslander algebras, contributing to a broader understanding of their algebraic properties and classification.
\end{abstract}

\vspace{-3mm}

\setcounter{tocdepth}{3}
\setcounter{secnumdepth}{3}
\tableofcontents


\section{Introduction}
String almost gentle algebras (abbreviated as SAG-algebras), a special class of string algebras, play an important role in representation theory and were first introduced by Green and Schroll in \cite{GS2018}. The systematic study of string algebras can be traced back to the work on finitely generated module categories over string algebras in \cite{BR1987}, where Butler and Ringel provided descriptions of indecomposable modules using strings and bands on the bound quivers of string algebras. Furthermore, by applying the Brauer-Thrall theorem (see, for example, \cite[Chapter IV, Section IV.5]{ASS2006}), it is understood that the representation types of string and gentle algebras are characterized by the existence of bands.

In \cite{P2019}, Plamondon shows that all (support) $\tau$-tilting finite gentle algebras are representation-finite, which partially answers the Brauer-Thrall Problem within the context of $\tau$-tilting theory—specifically, whether a $\tau$-tilting finite algebra is necessarily representation-finite. Building on these results, Mousavand investigated the relationship between representation types and $\tau$-tilting finiteness in biserial algebras in \cite{Mou2023}, providing examples of finite-dimensional algebras where the representation type and $\tau$-tilting finiteness do not coincide.
Furthermore, in \cite{LZHpre}, the authors offer an alternative description of gentle algebras using Gorenstein projective support $\tau$-tilting modules (abbreviated as GPS$\tau$-tilting modules), based on the work of \cite{Kal2015}. The concept of GPS$\tau$-tilting modules, introduced by Xie and Zhang in \cite{XZ2021}, refers to modules that are both Gorenstein projective and support $\tau$-tilting. The authors demonstrate that a gentle algebra, denoted as $\itLamb$, is representation-finite if and only if, for any GPS$\tau$-tilting module $M$, the endomorphism algebra $\End_{\itLamb}(M)$ is also representation-finite.
This result establishes a significant connection between Gorenstein projective modules, $\tau$-tilting modules, and the representation types of gentle algebras. The proof hinges on the fact that any Gorenstein projective module over a gentle algebra $\itLamb$ is isomorphic to $\alpha\itLamb$, where $\alpha$ is an arrow satisfying certain special conditions. Consequently, the Cohen-Macaulay Auslander algebra (abbreviated as CM Auslander algebra) $\itLamb^{\CMA}$ of $\itLamb$ takes the form $\End_{\itLamb}\big(\itLamb \oplus \bigoplus_{\alpha} \alpha\itLamb\big)$.
It is worth noting that, in \cite{CL2017,CL2019}, Chen and Lu revealed that the representation types of (skew-)gentle algebras and their CM Auslander algebras coincide. However, for an algebra $A = \kk\Q/\I$, not all $\alpha A$ are Gorenstein projective. Thus, this naturally raises the following question.

\begin{question} \label{quest:1}
Is there a subset $\R$ of the arrow set of $\Q$ such that the representation types of $A$ and $\End_A\left(A \oplus \bigoplus_{\alpha\in\R} \alpha A\right)$ coincide?
\end{question}

We will address the above questions in the case where $A$ is an SAG-algebra. Throughout this paper, we assume that $\kk$ is an algebraically closed field, and we define a quiver as a quadruple $\Q = (\Q_0, \Q_1, \s, \t)$, where $\Q_0$ is the set of vertices, $\Q_1$ is the set of arrows, and $\s$ and $\t$ are functions $\Q_1 \to \Q_0$ that assign to each arrow $a \in \Q_1$ its source and target, respectively.
Furthermore, we denote by $\Q_{\ell}$ the set of all paths of length $\ell$ (hence, $\Q_0$ naturally corresponds to the set of all paths of length zero, and $\Q_1$ to the set of all paths of length one). If $a$ and $b$ are arrows such that $\t(a) = \s(b)$, the composition of $a$ and $b$ is denoted by $ab$.
All algebras considered in this paper are finite-dimensional $\kk$-algebras, and for any algebra $A$, all modules under consideration are finitely generated right $A$-modules.

Let $A$ be an SAG-algebra with bound quiver $(\Q, \I)$. The main results of this paper are summarized as follows.

\begin{theorem}[{Theorem \ref{thm:main 1}}]
 There exists at least one subset $\R$ of $\Q_1$ {\rm (}note that the module $\alpha A$ with $\alpha \in \R$ may not be Gorenstein projective; see {\rm Example \ref{examp:SAG}}{\rm )} such that the following statements hold: \begin{itemize} \item[{\rm(1)}] The bound quiver $(\R(\Q), \R(\I))$ of $A_{\R} := \End_A\big(A \oplus \bigoplus_{\alpha\in\R} \alpha A\big)$ can be described by Steps 1--6 in {\rm Subsection \ref{subsect:R bound quiver}}; \item[{\rm(2)}] $A_{\R}$ is an SAG-algebra. \end{itemize} \end{theorem}

Indeed, the subset $\R$ in the above theorem is the set of certain left forbidden arrows on $(\Q, \I)$, referred to as a left forbidden arrow index; see Definition \ref{def:R}. In the case where $A$ is a gentle algebra, $\R$ can be equal to \begin{center} $\mathcal{G} = {\alpha \in \Q_1 \mid \alpha A \text{ is both non-projective and Gorenstein projective}}$ \end{center} or another left forbidden arrow index. In particular, when $\R = \mathcal{G}$, we have $A_{\mathcal{G}} = A_{\R} \cong A^{\CMA}$, as stated in \cite[Theorem 3.5]{CL2019}. The following result extends the findings of \cite[Theorem 3.5]{CL2019} to SAG-algebras.

\begin{theorem}[{Theorem \ref{thm:main 4}}]
Let $\C_1,\ldots,\C_t$ be perfect forbidden cycles on the bound quiver $(\Q,\I)$ of an SAG-algebra $A$.
Then $A_{\perR}$ is isomorphic to the CM-Auslander algebra $A^{\CMA}$ of $A$.
\end{theorem}

The following theorem provide some descriptions of the representation types of SAG-algebras.

\begin{theorem}
An SAG-algebra $A=\kk\Q/\I$ is representation-finite if and only if either of the following statements holds.
\begin{itemize}
  \item[{\rm(1)}] {\rm(Theorem \ref{thm:main 2})}
    There exists a left forbidden arrow index $\R$ such that $A_{\R}$ is representation-finite.
  \item[{\rm(2)}] {\rm(Corollary \ref{coro:main 3 pre})}
    For all left forbidden arrow indices $\R$, the $\R$-endomorphism algebras $A_{\R}$ is representation-finite.
  \item[{\rm(3)}] {\rm(Corollary \ref{coro:main 3-2})}
    The CM-Auslander algebra $A^{\CMA}$ of $A$ is representation-finite.
\end{itemize}
\end{theorem}

\section{String algebras, SAG-algebras, and their module categories}

\subsection{String algebras and SAG-algebras}

A {\defines monomial algebra} is a finite dimensional $\kk$-algebra which is Morita equivalent to $\kk\Q/\I$ such that $\I$ is generated by some paths of length $\geqslant 2$.
String algebras are special monomial algebras. In this part, we recall some concepts for string algebras.

Let $\Q$ be a quiver and $\I$ be an ideal of $\kk\Q$ such that $\kk\Q/\I$ is a monomial algebra.
We say a bound quiver $(\Q, \I)$ is a {\defines string pair} if it satisfies the following conditions.
\begin{itemize}
  \item[(S1)$_{\color{white}\Right}$]
     Any vertex of $\Q$ is the source of at most two arrows and the target of at most two arrows.
  \item[(S2)$_\Right$] For each arrow $\alpha:x\to y$, there is at most one arrow $\beta$
   whose source $\s(\beta)$ is $y$ such that $\alpha\beta\notin\I$.
  \item[(S2)$_\Left$] For each arrow $\alpha:x\to y$, there is at most one arrow $\gamma$
   whose target $\t(\gamma)$ is $x$ such that $\gamma\alpha\notin\I$.
\end{itemize}
We say that a bound quiver $(\Q, \I)$ of a monomial algebra is a {\defines almost gentle pair} if it satisfies the following conditions.
\begin{itemize}
  \item[(AG1)] (S2)$_\Right$ and (S2)$_\Left$ holds.
  \item[(AG2)] All generators of the ideal $\I$ are paths of length two.
\end{itemize}

Now we recall the definitions of string algebra, almost gentle algebra, and string almost gentle algebra.

\begin{definition}\rm
Let $A$ be a finite-dimensional algebra. We call that $A$ is a:
\begin{itemize}
  \item[(1)] {\defines string {\rm(}resp., almost gentle{\rm)} algebra}, if $A$ is Morita equivalent to $\kk\Q/\I$ such that $(\Q, \I)$ is a string {\rm(}resp., almost gentle{\rm)} pair;
  \item[(2)] {\defines string almost gentle algebra} (=SAG-algebra), if $A$ is both string and almost gentle.
\end{itemize}
\end{definition}

\begin{example} \label{examp:string} \rm
Let $A=\kk\Q/\I$ be an algebra whose bound quiver $(\Q,\I)$ is shown in \Pic \ref{fig in examp:string}, where
\begin{center}
  $\I = \langle ab, bc, ca, dd', ee', ff',$
  $a'b', b'c', c'a', $

  $e'f, e'c', f'd, f'a', d'e, d'b',$
  $a'eb, b'fc, c'da\rangle$.
\end{center}
\begin{figure}[htbp]
\centering
\begin{tikzpicture}[scale=2]
\draw[   red][dotted][line width = 0.8pt][rotate around={  0:(0,0)}]
  ( 0.51, 0.78) to[out=-60,in=60] ( 0.51,-0.78);
\draw[yellow][dotted][line width = 0.8pt][rotate around={120:(0,0)}]
  ( 0.51, 0.78) to[out=-60,in=60] ( 0.51,-0.78);
\draw[  blue][dotted][line width = 0.8pt][rotate around={240:(0,0)}]
  ( 0.51, 0.78) to[out=-60,in=60] ( 0.51,-0.78);
\draw[   red][dotted][line width = 0.8pt][rotate around={  0:(0,0)}]
  ( 1.50,-0.15) -- ( 1.00, 1.00);
\draw[yellow][dotted][line width = 0.8pt][rotate around={120:(0,0)}]
  ( 1.50,-0.15) -- ( 1.00, 1.00);
\draw[  blue][dotted][line width = 0.8pt][rotate around={240:(0,0)}]
  ( 1.50,-0.15) -- ( 1.00, 1.00);
\draw[   red][dashed][line width = 1.2pt][rotate around={  0:(0,0)}]
  ( 2.15,-0.15) arc (180:225:0.55);
\draw[yellow][dashed][line width = 1.2pt][rotate around={120:(0,0)}]
  ( 2.15,-0.15) arc (180:225:0.55);
\draw[  blue][dashed][line width = 1.2pt][rotate around={240:(0,0)}]
  ( 2.15,-0.15) arc (180:225:0.55);
\draw[   red][dotted][line width = 1.6pt][rotate around={  0:(0,0)}]
  (2.6,0) arc(0:115:2.6) -- (-0.25, 0.83) arc(120:215:1.1);
\draw[yellow][dotted][line width = 2.6pt][rotate around={120:(0,0)}]
  (2.6,0) arc(0:115:2.6) -- (-0.25, 0.83) arc(120:215:1.1);
\draw[blue!55][dotted][line width = 3.6pt][rotate around={240:(0,0)}]
  (2.6,0) arc(0:115:2.6) -- (-0.25, 0.83) arc(120:215:1.1);
\draw[   red][dotted][line width = 0.8pt][rotate around={  0:(0,0)}]
  (2.70,-0.15-0.60) arc( -90:  90:0.60);
\draw[   red][dashed][line width = 0.8pt][rotate around={  0:(0,0)}]
  (2.70,-0.15+0.75) arc(  90:-136:0.75);
\draw[yellow][dotted][line width = 0.8pt][rotate around={120:(0,0)}]
  (2.70,-0.15-0.60) arc( -90:  90:0.60);
\draw[yellow][dashed][line width = 0.8pt][rotate around={120:(0,0)}]
  (2.70,-0.15+0.75) arc(  90:-136:0.75);
\draw[  blue][dotted][line width = 0.8pt][rotate around={240:(0,0)}]
  (2.70,-0.15-0.60) arc( -90:  90:0.60);
\draw[  blue][dashed][line width = 0.8pt][rotate around={240:(0,0)}]
  (2.70,-0.15+0.75) arc(  90:-136:0.75);
\draw[line width = 0.8pt][->] [rotate around={  0:(0,0)}]
  (1,0) arc(0:105:1);
\draw[line width = 0.8pt][->] [rotate around={120:(0,0)}]
  (1,0) arc(0:105:1);
\draw[line width = 0.8pt][->] [rotate around={240:(0,0)}]
  (1,0) arc(0:105:1);
\draw[line width = 0.8pt][->] [rotate around={  0:(0,0)}]
  (2.5,-0.15) -- (1.2,-0.15);
\draw[line width = 0.8pt][->] [rotate around={120:(0,0)}]
  (2.5,-0.15) -- (1.2,-0.15);
\draw[line width = 0.8pt][->] [rotate around={240:(0,0)}]
  (2.5,-0.15) -- (1.2,-0.15);
\draw[line width = 0.8pt][->] [rotate around={  0:(0,0)}]
  (1.1,-0.05) to[out=80,in=0] (-1.1,2.3);
\draw[line width = 0.8pt][->] [rotate around={120:(0,0)}]
  (1.1,-0.05) to[out=80,in=0] (-1.1,2.3);
\draw[line width = 0.8pt][->] [rotate around={240:(0,0)}]
  (1.1,-0.05) to[out=80,in=0] (-1.1,2.3);
\draw[line width = 0.8pt][->] [rotate around={  0:(0,0)}]
  (2.7,0) arc(0:115:2.7);
\draw[line width = 0.8pt][->] [rotate around={120:(0,0)}]
  (2.7,0) arc(0:115:2.7);
\draw[line width = 0.8pt][->] [rotate around={240:(0,0)}]
  (2.7,0) arc(0:115:2.7);
{\tiny
\draw[rotate around={  0:(0,0)}] ( 1.00,-0.15) node{$1$};
\draw[rotate around={120:(0,0)}] ( 1.00,-0.15) node{$2$};
\draw[rotate around={240:(0,0)}] ( 1.00,-0.15) node{$3$};
\draw[rotate around={  0:(0,0)}] ( 2.70,-0.15) node{$4$};
\draw[rotate around={120:(0,0)}] ( 2.70,-0.15) node{$5$};
\draw[rotate around={240:(0,0)}] ( 2.70,-0.15) node{$6$};
\draw[rotate around={  0:(0,0)}] ( 0.45, 0.78) node{$a$};
\draw[rotate around={120:(0,0)}] ( 0.45, 0.78) node{$b$};
\draw[rotate around={240:(0,0)}] ( 0.45, 0.78) node{$c$};
\draw[rotate around={  0:(0,0)}] ( 1.85, 0.05) node{$d$};
\draw[rotate around={120:(0,0)}] ( 1.85, 0.05) node{$e$};
\draw[rotate around={240:(0,0)}] ( 1.85, 0.05) node{$f$};
\draw[rotate around={  0:(0,0)}] ( 0.85, 1.47) node{$d'$};
\draw[rotate around={120:(0,0)}] ( 0.85, 1.47) node{$e'$};
\draw[rotate around={240:(0,0)}] ( 0.85, 1.47) node{$f'$};
\draw[rotate around={  0:(0,0)}] ( 1.45, 2.51) node{$a'$};
\draw[rotate around={120:(0,0)}] ( 1.45, 2.51) node{$b'$};
\draw[rotate around={240:(0,0)}] ( 1.45, 2.51) node{$c'$};
}
\end{tikzpicture}
\caption{The bound quiver of the string algebra given in Example \ref{examp:string}}
\label{fig in examp:string}
(The dashed lines represent the relations in $\I$)
\end{figure}
Then $A$ is a string algebra. In this case, $A$ is not a SAG-algebra because the lengths of relations $a'eb$, $b'fc$, and $c'da$ are $3$.
\end{example}

\subsection{The module categories of string algebras}

In \cite{BR1987}, Butler and Ringel have described all indecomposable modules over string algebra.
In this subsection we recall strings, bands, string modules, and band modules.

For any arrow $a\in \Q_1$, we denote by $a^{-1}$ the {\defines formal inverse} of $a$. Then $\s(a^{-1})=\t(a)$ and $\t(a^{-1})=\s(a)$.
Define $\Q_1^{-1}:=\{a^{-1}\mid a\in \Q_1\}$ be the set of all formal inverses of arrows.
Then any path $p=a_1a_2\cdots a_\ell$ on a bound quiver $(\Q, \I)$ naturally provides a formal inverse path
$p^{-1} = a_{\ell}^{-1}a_{\ell-1}^{-1}\cdots a_1^{-1}$ of $p$.
In particular, for any path $\e_v$ of length zero corresponding to $v\in \Q_0$, we define $\e_v^{-1} = \e_v$.

\begin{definition}\rm
A {\defines string} on a bound quiver $(\Q, \I)$ is a sequence $s=(\wp_1, \wp_2, \ldots, \wp_n)$,
where $\wp_{i} = a_{i,1}\cdots a_{i,l_i}$, $1\le i\le n$, and $a_{i,j}\in \Q_1\cup\Q_1^{-1}$, $1\le j\le l_i$, such that:
\begin{itemize}
  \item[(S]\hspace{-5pt}tr1)
    for any $1\le i\le n$, $\wp_i$ or $\wp_i^{-1}$ is a path on $(\Q, \I)$;

  \item[(S]\hspace{-5pt}tr2)
    if $\wp_i$ is a path, then $\wp_{i+1}$ is a formal inverse path,
    and $a_{i,l_i}\ne a_{i+1,1}^{-1}$;

  \item[(S]\hspace{-5pt}tr3)
    if $\wp_i$ is a formal inverse path, then $\wp_{i+1}$ is a path,
    and $a_{i,l_i}^{-1}\ne a_{i+1,1}$;

  \item[(S]\hspace{-5pt}tr4)
    $\t(\wp_{i})=\s(\wp_{i+1})$ holds for all $1\le i\le n-1$,
    which are called {\defines turning points}.
\end{itemize}
A {\defines band} $b=(\wp_1, \wp_2, \ldots, \wp_n)$ is a string such that:
\begin{itemize}
  \item[(B]
    \hspace{-9pt}
    and1)
    $\t(\wp_n)=\s(\wp_1)$, and if $\wp_n$ and $\wp_1$ are paths then $\wp_n\wp_1\notin \I$, if $\wp_n$ and $\wp_1$ are formal inverse paths then $(\wp_n\wp_1)^{-1} \notin \I$;

  \item[(B]
    \hspace{-9pt}
    and2)
    $b$ is not a non-trivial power of some string, i.e.,
    there is no string $s$ such that $b=s^m$ for some $m\ge 2$.
\end{itemize}
A vertex $v$ on a string $s$ is called a {\defines source} if one of the following condition holds:
\begin{itemize}
  \item[(1)] $v$ is a turning point $\t(\wp_{i})=\s(\wp_{i+1})$ such that $\wp_{i}$ is a formal inverse path and $\wp_{i+1}$ is a path;
  \item[(2)] $\wp_1$ is a path, and $v=\s(s)=\s(\wp_1)$;
  \item[(3)] $\wp_n$ is a formal inverse path, and $v=\t(s)=\t(\wp_n)$.
\end{itemize}
We can define {\defines sink} by dual way.
\end{definition}

\begin{remark}\rm
We can define the {\defines substring} by removing interconnected arrows on both sides of string.
\end{remark}

\begin{definition} \rm \
\begin{itemize}
  \item[(1)] $s$ is called a {\defines trivial string} if it is an empty;
  \item[(2)] two strings $s$ and $s'$ are called {\defines equivalent} if $s'=s$ or $s'=s^{-1}$;
  \item[(3)] two bands $b=\alpha_1\cdots\alpha_n$ and $b'=\alpha_1'\cdots\alpha_t'$ are called {\defines equivalent} if $b[t]=b'$ or $b[t]^{-1}=b'$, where $b[t]=\alpha_{1+t}\cdots\alpha_{n}\alpha_1\cdots\alpha_{1+t-1}$.
\end{itemize}
We denote by $\str(A)$ (resp., $\band(A)$) the set of all equivalent classes of strings (resp., bands) on the bound quiver of $A$, respectively.
\end{definition}

The following result is first shown by Butler and Ringel.

\begin{theorem}[{Butler-Ringel \cite[Section 3]{BR1987}}] \label{thm:BR1987}
All indecomposable objects in category $\modcat(A)$ of a string algebra $A$ can be described by the following bijection
  \[ \M: \str(A) \cup (\band(A)\times\mathscr{J}) \to \ind(\modcat(A)), \]
where $\ind(\modcat(A))$ is the set of all isoclasses of indecomposable $A$-modules and $\mathscr{J}$ is the set of all Jordan block with non-zero eigenvalue.
\end{theorem}

Notice that we can define strings and bands on any monomial pair $(\Q,\I)$
and each indecomposable module corresponded by string and band
is called a {\defines string module} and {\defines band module}, respectively.
However, the set $\ind(\modcat A)$, where $A=\kk\Q/\I$, of all isoclasses of indecomposable $A$-modules

A string $s$ can be written as
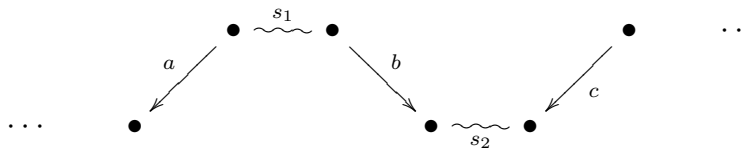
\begin{center}
  $\xymatrix{
   & & \bullet \ar@{~}[r]^{s_1} \ar[ld]_{a} & \bullet \ar[rd]^{b} & & & \bullet & \cdots \\
  \cdots & \bullet & & & \bullet \ar@{~}[r]_{s_2} & \bullet \ar@{<-}[ru]_{c} & & &\\
  }$
\end{center}
up to equivalence by using arrows $a\in\Q_1$ on bound quiver $(\Q,\I)$.
In this case, the substring $s_1$ (resp., $s_2$) is said to be the {\defines factor substring}
(resp., {\defines image substring}) of $s$ (respect to the pair $(a,b)$ (resp., $(b,c)$)).
In particular, if $a$ does not exist, that is, $s$ is of the form
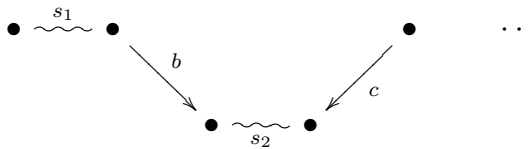
\begin{center}
  $\xymatrix{
   & & \bullet \ar@{~}[r]^{s_1} & \bullet \ar[rd]^{b} & & & \bullet & \cdots \\
   & & & & \bullet \ar@{~}[r]_{s_2} & \bullet \ar@{<-}[ru]_{c} & & &\\
  }$
\end{center}
then $s_1$ is said to be the factor substring of $s$ respect to the pair $(0,b)$.
We can define factor substring of $s$ respect to the pair $(a,0)$,
image substring of $s$ respect to the pair $(b,0)$,
and image substring of $s$ respect to the pair $(0,c)$ by similar way.

Factor and image substrings can be used to describe the homomorphisms between two string modules as the following
result, see for example \cite[Theorem in page 191]{Kra1991} and \cite[Chapter 2, Section 2, 2.4.2]{Lak2016}.

\begin{theorem} \label{thm:Kra1991} \
\begin{itemize}
\item[\rm(1)]
For two string modules corresponded by strings $s_1$ and $s_2$,
$\Hom_A(\M(s_2),$ $\M(s_1)) \ne 0$
if and only if there is a factor substring $q$ of $s_2$ and an image substring $p$ of $s_1$ such that
$q$ and $p$ coincide.

\item[\rm(2)]
Furthermore, take two string $s_1$ and $s_2$ as following:
\begin{center}
\begin{tikzpicture}[xscale=0.75]
\draw [->] (-0.8, 0.8) -- (-0.2, 0.2);
\draw [->] ( 1.8, 0.8) -- ( 1.2, 0.2);
\fill ( 0, 0) circle (2pt) ( 1, 0) circle (2pt)
      (-1, 1) circle (2pt) ( 2, 1) circle (2pt);
\draw [domain=0.15:0.85,samples=160,smooth]
      plot (\x,{0.05*cos(\x*2*pi*300)});
\draw [shift={(2,1)}][dotted][line width=1pt]
      (0.2,0) -- (0.8,0);
\draw [shift={(-4,0)}][->] (-0.8, 0.8) -- (-0.2, 0.2);
\draw [shift={(-4,0)}][->] ( 1.8, 0.8) -- ( 1.2, 0.2);
\fill [shift={(-4,0)}]
      ( 0, 0) circle (2pt) ( 1, 0) circle (2pt)
      (-1, 1) circle (2pt) ( 2, 1) circle (2pt);
\draw [shift={(-4,0)}][domain=0.15:0.85,samples=160,smooth]
      plot (\x,{0.05*cos(\x*2*pi*300)});
\draw [dotted][line width=1pt](-1.8,1)--(-1.2,1);
\draw [shift={(-6,1)}][dotted][line width=1pt]
      (0.2,0) -- (0.8,0);
\draw [shift={(4,0)}][->] (-0.8, 0.8) -- (-0.2, 0.2);
\draw [shift={(4,0)}][->] ( 1.8, 0.8) -- ( 1.2, 0.2);
\fill [shift={(4,0)}]
      ( 0, 0) circle (2pt) ( 1, 0) circle (2pt)
      (-1, 1) circle (2pt) ( 2, 1) circle (2pt);
\draw [shift={(4,0)}][domain=0.15:0.85,samples=160,smooth]
      plot (\x,{0.05*cos(\x*2*pi*300)});
\draw [shift={(8,0)}][->] (-0.8, 0.8) -- (-0.2, 0.2);
\draw [shift={(8,0)}][->] ( 1.8, 0.8) -- ( 1.2, 0.2);
\fill [shift={(8,0)}]
      ( 0, 0) circle (2pt) ( 1, 0) circle (2pt)
      (-1, 1) circle (2pt) ( 2, 1) circle (2pt);
\draw [shift={(8,0)}][domain=0.15:0.85,samples=160,smooth]
      plot (\x,{0.05*cos(\x*2*pi*300)});
\draw [shift={(6,1)}][dotted][line width=1pt]
      (0.2,0) -- (0.8,0);
\draw [shift={(10,1)}][dotted][line width=1pt]
      (0.2,0) -- (0.8,0);
{\tiny
\draw[shift={(-4.0,0)}] (0.5,0) node[above]{$p_{1}$};
\draw[shift={( 0.0,0)}] (0.5,0) node[above]{$p_{2}$};
\draw[shift={( 2.0,0)}] (0.5,0) node[above]{$\cdots$};
\draw[shift={( 4.0,0)}] (0.5,0) node[above]{$p_{m-1}$};
\draw[shift={( 8.0,0)}] (0.5,0) node[above]{$p_{m}$};
\draw[shift={(-4.0,0)}] (-0.5,0.5) node[left]{$a_{1}$};
\draw[shift={(-2.0,0)}] (-0.5,0.5) node[left]{$a_{1}'$};
\draw[shift={( 0.0,0)}] (-0.5,0.5) node[left]{$a_{2}$};
\draw[shift={( 1.0,0)}] ( 0.5,0.5) node[left]{$a_{2}'$};
\draw[shift={( 4.0,0)}] (-0.5,0.5) node[left]{$a_{m-1}'$};
\draw[shift={( 5.5,0)}] ( 0.5,0.5) node[left]{$a_{m-1}$};
\draw[shift={( 8.0,0)}] (-0.5,0.5) node[left]{$a_{m}'$};
\draw[shift={( 9.0,0)}] ( 0.5,0.5) node[left]{$a_{m}$};
}
\draw (-7,0) node{$s_1=$};
\end{tikzpicture}

\

\begin{tikzpicture}[xscale=0.75]
\draw [<-] (-0.8,-0.8) -- (-0.2,-0.2);
\draw [<-] ( 1.8,-0.8) -- ( 1.2,-0.2);
\fill ( 0, 0) circle (2pt) ( 1, 0) circle (2pt)
      (-1,-1) circle (2pt) ( 2,-1) circle (2pt);
\draw [domain=0.15:0.85,samples=160,smooth]
      plot (\x,{0.05*cos(\x*2*pi*300)});
\draw [shift={(2,-1)}][dotted][line width=1pt]
      (0.2,0) -- (0.8,0);
\draw [shift={(-4,0)}][<-] (-0.8,-0.8) -- (-0.2,-0.2);
\draw [shift={(-4,0)}][<-] ( 1.8,-0.8) -- ( 1.2,-0.2);
\fill [shift={(-4,0)}]
      ( 0,-0) circle (2pt) ( 1,-0) circle (2pt)
      (-1,-1) circle (2pt) ( 2,-1) circle (2pt);
\draw [shift={(-4,0)}][domain=0.15:0.85,samples=160,smooth]
      plot (\x,{0.05*cos(\x*2*pi*300)});
\draw [dotted][line width=1pt](-1.8,-1)--(-1.2,-1);
\draw [shift={(-6,-1)}][dotted][line width=1pt]
      (0.2,0) -- (0.8,0);
\draw [shift={(4,0)}][<-] (-0.8,-0.8) -- (-0.2,-0.2);
\draw [shift={(4,0)}][<-] ( 1.8,-0.8) -- ( 1.2,-0.2);
\fill [shift={(4,0)}]
      ( 0,-0) circle (2pt) ( 1,-0) circle (2pt)
      (-1,-1) circle (2pt) ( 2,-1) circle (2pt);
\draw [shift={(4,0)}][domain=0.15:0.85,samples=160,smooth]
      plot (\x,{0.05*cos(\x*2*pi*300)});
\draw [shift={(8,0)}][<-] (-0.8,-0.8) -- (-0.2,-0.2);
\draw [shift={(8,0)}][<-] ( 1.8,-0.8) -- ( 1.2,-0.2);
\fill [shift={(8,0)}]
      ( 0, 0) circle (2pt) ( 1, 0) circle (2pt)
      (-1,-1) circle (2pt) ( 2,-1) circle (2pt);
\draw [shift={(8,0)}][domain=0.15:0.85,samples=160,smooth]
      plot (\x,{0.05*cos(\x*2*pi*300)});
\draw [shift={(6,-1)}][dotted][line width=1pt]
      (0.2,0) -- (0.8,0);
\draw [shift={(10,-1)}][dotted][line width=1pt]
      (0.2,0) -- (0.8,0);
{\tiny
\draw[shift={(-4.0,0)}] (0.5,0) node[below]{$q_{1}$};
\draw[shift={(-2.0,0)}] (0.5,0) node[below]{$\cdots$};
\draw[shift={( 0.0,0)}] (0.5,0) node[below]{$q_{2}$};
\draw[shift={( 2.0,0)}] (0.5,0) node[below]{$\cdots$};
\draw[shift={( 4.0,0)}] (0.5,0) node[below]{$q_{m-1}$};
\draw[shift={( 6.0,0)}] (0.5,0) node[below]{$\cdots$};
\draw[shift={( 8.0,0)}] (0.5,0) node[below]{$q_{m}$};
\draw[shift={(-4.0,0)}] (-0.5,-0.5) node[left]{$b_{1}$};
\draw[shift={(-2.0,0)}] (-0.5,-0.5) node[left]{$b_{1}'$};
\draw[shift={( 0.0,0)}] (-0.5,-0.5) node[left]{$b_{2}$};
\draw[shift={( 1.0,0)}] ( 0.5,-0.5) node[left]{$b_{2}'$};
\draw[shift={( 4.0,0)}] (-0.5,-0.5) node[left]{$b_{m-1}$};
\draw[shift={( 5.5,0)}] ( 0.5,-0.5) node[left]{$b_{m-1}'$};
\draw[shift={( 8.0,0)}] (-0.5,-0.5) node[left]{$b_{m}$};
\draw[shift={( 9.0,0)}] ( 0.5,-0.5) node[left]{$b_{m}'$};
}
\draw (-7,0) node{$s_2=$};
\end{tikzpicture}
\end{center}
where all $p_r$ are image substrings respect to $(a_r,a_{r}')$ of $s_1$,
and all $q_r$ are factor substrings respect to $(b_r,b_{r}')$ of $s_2$  $(1\leqslant r\leqslant m)$.
If
\begin{itemize}
  \item $p_{1}=q_{1}$, $p_{2}=q_{2}$, $\ldots$, $p_{m}=q_{m}$,
  \item and for other image substring $p$ of $s_1$ which does not is a substring of any $p_r$,
    there is no factor substring $q$ of $s_2$ such that $p=q$,
\end{itemize}
then
\[ \dim_{\kk}\Hom_A(\M(s_2), \M(s_1)) = m. \]
{\rm(}All pairs $(p_r,q_r)$ describe the basis of $\Hom_A(\M(s_2), \M(s_1))$ as $\kk$-linear space.{\rm)}
\end{itemize}
\end{theorem}

\subsection{Cycles}

A path $p = a_1\cdots a_n$ on a quiver $\Q$ is said to be {\defines forbidden} if $a_ia_{i+1} \in \I$ holds for all $1\le i\le n-1$.
The arrows $a_1, \ldots, a_{n-1}$ are called {\defines left forbidden arrows}
and the arrows $a_2, \ldots, a_n$ are called {\defines right forbidden arrows}.

Next, we recall the definition of forbidden cycle.

\begin{definition} \rm
Let $\overline{\Q}$ be the underlying graph\footnote{Recall that the {\defines underlying graph} $\overline{\Q}$ of $\Q$ is obtained from $\Q$ by forgetting the orientation of the arrows. Each $\overline{\alpha}$, the arrow $\alpha$ forgetting orientation, is called an {\defines edges} of $\overline{\Q}$.} of $\Q$.
A {\defines cycle} $\C$ (of length $n$) on $n$ vertices $v_1,\ldots, v_n \in \Q_0$ is a sequence of $n$ edges $\overline{c}_1, \ldots, \overline{c}_n$ of $\overline{\Q}$ such that the vertices of $\C$ can be arranged in a cyclic sequence in such a way that two vertices $v_i$ and $v_{i+1}$ are adjacent connected by the arrow $c_i$ if they are consecutive in the sequence, and are nonadjacent otherwise (the indices $i$ are taken modulo $n$ if necessary).
An {\defines oriented cycle} is a cycle $\C=c_1\cdots c_n$ with $\t(c_{i})=\s(c_{i+1})$  $(1\le i<n)$ such that $\t(\C)=\t(c_n)=\s(c_1)=\s(\C)$ holds.
Furthermore, $\C$ is called a {\defines forbidden cycle} if there are relations $r_0, r_1, \cdots, r_{d-1}$ of $\I$ such that $c_1c_2$, $\ldots$, $c_{n-1}c_n$, $c_nc_1$ $\in\C$.
A {\defines cycle without relation} is a cycle $\C$ such that all paths on $\C$ are not in $\I$.
\end{definition}

\begin{remark} \rm \
\begin{itemize}
  \item[(1)] Forbidden paths are introduced by Avella-Alaminos and Geiss in \cite{AG2008} which are used to describe AG-invariants of gentle algebras.
    The terminology ``forbidden cycle'' and ``forbidden arrow'' come from forbidden path.

  \item[(2)] Each cycle without relation provide a band.
\end{itemize}
\end{remark}

\section{The module $\alpha A$}
In this section, we consider the $A$-module $\alpha A$, where $\alpha$ is an arrow on the string pair $(\Q, \I)$.

\subsection{$\alpha A$ is an indecomposable module}

We introduce the module $\alpha A$ and show that it is an indecomposable module in this part.

\begin{lemma}\label{lemm:alphaA}
For any arrow $\alpha \in \Q_1$ on a string pair $(\Q, \I)$, we have:
\begin{itemize}
  \item[\rm(1)]
    $\alpha A \leqslant_{\oplus} \rad(e_{\s(\alpha)}A)$,
    where $e_{\s(\alpha)}$ is the idempotent corresponded by $\s(\alpha)$, and
  \item[\rm(2)] $\alpha A$ is an indecomposable module.
\end{itemize}
\end{lemma}

\begin{proof}
First of all, we show that there exists an injection
\begin{align}
 \sigma:\ & \alpha A=\sum\limits_{\wp \in \Q_{\ge 0}  \atop \t(\alpha)=\s(\wp)} \kk \alpha\wp
 \xymatrix{\ar@{^(->}[r]^{\subseteq}&}
 e_{\s(\alpha)}A
 = \sum\limits_{\tilde{\wp} \in \Q_{\ge 0} \atop \s(\tilde{\wp}) = \s(\alpha)} \kk\tilde{\wp}.   \nonumber
\end{align}

Any path $\alpha\wp$ in $\alpha A$ is a path with source $\s(\alpha)$.
By the definition of string pair, $\s(\alpha)$ is a source of at most two arrows,
and then we obtain two cases as follows.
\begin{itemize}
  \item[(1)] There are two arrows $a_1$ and $a_1'$ such that $\s(a_1)=\s(a_1')=\s(\alpha)$
    ($\alpha$ equals to either $a_1$ or $a_1'$).
    In this case, $e_{\s(\alpha)}A$ is the indecomposable module corresponding to some string which is of the form
      \[ \bullet \larrow{a_m'} \bullet \cdots \bullet \larrow{a_2'} \bullet \larrow{a_1'}
         \bullet \rarrow{a_1} \bullet \rarrow{a_2} \bullet \cdots \bullet \rarrow{a_n} \bullet \]
    and satisfies the following conditions:
    \begin{itemize}
      \item $\t(a_m')$ is a sink point of $\Q$, or there is an integer $1\leqslant i\leqslant m$ such that $a_i'a_{i+1}'\cdots a_m'a_{m+1}' \in \I$ holds for any arrow $a_{m+1}'$ with source $\s(a_{m+1}') = \t(a_m')$;
      \item $\t(a_n)$ is a sink point of $\Q$, or there is an integer $1\leqslant j \leqslant n$ such that $a_ja_{j+1}\cdots a_na_{n+1}\in \I$ holds for any arrow $a_{n+1}$ with source $\s(a_{n+1}) = \t(a_n)$.
     \end{itemize}
    Without loss of generality, assume that $\alpha=a_1$, then $\alpha A$ is the module corresponding to the string
    \[\bullet \rarrow{a_2} \bullet \cdots \bullet\rarrow{a_n} \bullet\]
    which is a direct summand of $\rad (e_{\s(\alpha)}A)$.

  \item[(2)] The arrow $\alpha$, written as $a$, is a unique arrow with source $\s(\alpha)$.
    In this case, $e_{\s(\alpha)}A$ is the indecomposable module corresponding to some string which is of the form in this case
      \[ \bullet \rarrow{a_1} \bullet \rarrow{a_2} \bullet \cdots \bullet \rarrow{a_n} \bullet. \]
    Thus, $\alpha A \le_{\oplus} \rad (e_{\s(\alpha)}A)$ can be given by the string $a_2\cdots a_n$ corresponding to $\alpha A$.
    This case can be seen as the case (1) with $m=0$.
\end{itemize}
By the above two cases, it is easy to see that $\alpha A$ is an indecomposable module.
\end{proof}

\subsection{Homomorphisms starting (resp. ending) with $\alpha A$}

A module is said to be an {\defines {\rm(}indecomposable{\rm)} arrowed module} if it isomorphic to $\alpha A$ for some $\alpha\in \Q_1$.
Let $\arrmod(A)$ be the set of all arrowed modules. The following lemma shows that
any homomorphism $h_{\alpha}$ induced by $\alpha$ between two indecomposable projective modules
$P(\s(\alpha))$ and $P(\t(\alpha))$ is a morphism crossing $\alpha A$.

\begin{lemma} \label{lemm:decomp}
For arbitrary arrow $\alpha\in\Q_1$ with source $\s(\alpha)=v$ and target $\t(\alpha)=w$, the morphism $h_{\alpha}: P(w) \to P(v)$ induced by $\alpha$ has a decomposition
\[\xymatrix{
P(w) \ar[rr]^{h_{\alpha}} \ar[rd]_{g} & & P(v).  \\
& \alpha A \ar[ru]_f &
}\]
\end{lemma}

\begin{proof}
For arbitrary $a\in A$, the homomorphism $h_{\alpha}$ induced by $\alpha\in\Q_1$ sends any $\e_wa \in \e_w A$ to $\alpha\cdot \e_w a = \e_v\cdot \alpha a \in \e_vA$.
It follows a decomposition
$\xymatrix@C=0.6cm{ \e_wA \ar@{->}[r]^{g} & \alpha A \ar@{->}[r]^{f} & \e_vA }$
of $h_{\alpha}$ satisfying
$\xymatrix{ \e_wa \ar@{|->}[r]^{g} & \alpha a \ar@{|->}[r]^{f} & \e_v\cdot \alpha a }$
as required.
\end{proof}

\begin{lemma}
\label{lemm:decomp-eA}
Keep the notations from Lemma \ref{lemm:decomp}.
The homomorphisms $f$ and $g$ can not be decomposed through any indecomposable projective module $P(u)$ $(\not\cong P(w))$.
\end{lemma}

\begin{proof}
Assume that the string corresponding to $P(w)$ is
\[ \M^{-1}(P(w)) = \ w_{m_0}' \larrow{a_{m_0-1}'} \cdots \larrow{a_2'} w_2' \larrow{a_1'}
  w_1 \rarrow{a_1} w_2 \rarrow{a_2} \cdots \rarrow{} w_{m-1} \rarrow{a_{m-1}} w_m. \]
Here, $w_1=w$. Then the string corresponding to $P(v)$ is of the form
\[ \M^{-1}(P(v)) = \ v_n \larrow{b_{n-1}} \cdots \larrow{b_2} v_2 \larrow{b_1} v_1 \rarrow{\alpha} w_1 \rarrow{} \cdots. \]
By the definition of string algebra, we have $\alpha a_1' \in \I$, see \Pic \ref{fig:decomp-eA-1}.

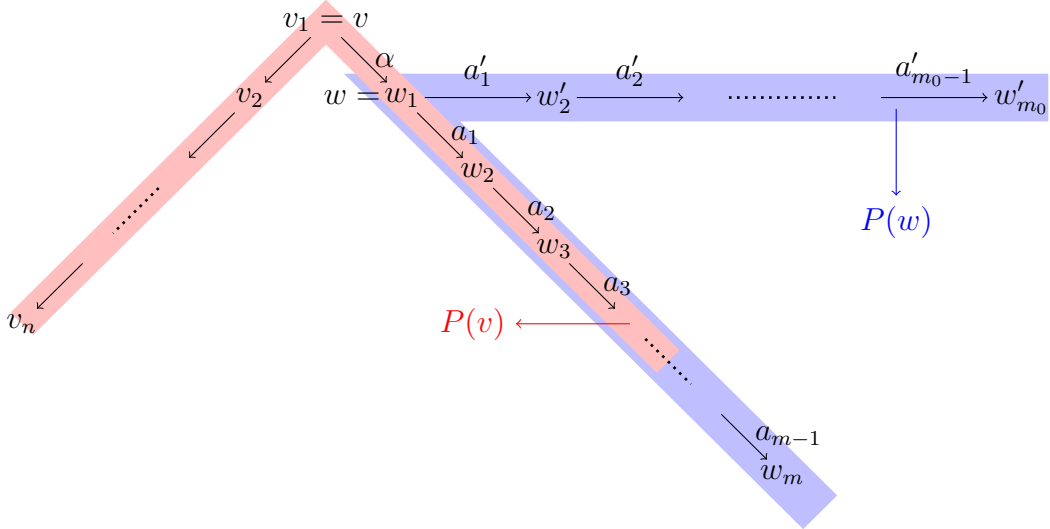
\begin{figure}[htbp]
\centering
\begin{tikzpicture}
%
\draw[blue!25][line width=18pt] (8.5,0) -- (0,0) -- (5.5,-5.5);
\draw[blue][->] (6.5,-0.15) -- (6.5,-1.3);
\draw[blue] (6.5,-1.3) node[below]{$P(w)$};
\draw[red!25][line width=12pt] (-5,-3) -- (-1,1) -- (3.5,-3.5);
\draw[red][->] (3.,-3.)--(1.5,-3);
\draw[red] (1.5,-3) node[left]{$P(v)$};
\draw                      (0,0) node[left]{$w=\ $};
\draw    [shift={( 0, 0)}] (0,0) node{$w_1$};
\draw    [shift={( 1,-1)}] (0,0) node{$w_2$};
\draw    [shift={( 2,-2)}] (0,0) node{$w_3$};
\draw    [shift={( 5,-5)}] (0,0) node{$w_m$};
\draw    [shift={( 2, 0)}] (0,0) node{$w_2'$};
\draw    [shift={( 8, 0)}] (0.15,0) node{$w_{m_0}'$};
\draw[->][shift={( 0, 0)}] (0.2,-0.2) -- (0.8,-0.8);
\draw[->][shift={( 1,-1)}] (0.2,-0.2) -- (0.8,-0.8);
\draw[->][shift={( 2,-2)}] (0.2,-0.2) -- (0.8,-0.8);
\draw    [shift={( 3,-3)}] [line width=1pt][dotted]
                           (0.2,-0.2) -- (0.8,-0.8);
\draw[->][shift={( 4,-4)}] (0.2,-0.2) -- (0.8,-0.8);
\draw    [shift={( 0, 0)}] (0.5,-0.5) node[right]{$a_1$};
\draw    [shift={( 1,-1)}] (0.5,-0.5) node[right]{$a_2$};
\draw    [shift={( 2,-2)}] (0.5,-0.5) node[right]{$a_3$};
\draw    [shift={( 4,-4)}] (0.5,-0.5) node[right]{$a_{m-1}$};
\draw[->][shift={( 0, 0)}] (0.3, 0.0) -- (1.7, 0.0);
\draw[->][shift={( 2, 0)}] (0.3, 0.0) -- (1.7, 0.0);
\draw    [shift={( 4, 0)}] [line width=1pt][dotted]
                           (0.3, 0.0) -- (1.7, 0.0);
\draw[->][shift={( 6, 0)}] (0.3, 0.0) -- (1.7, 0.0);
\draw    [shift={( 0, 0)}] (1.0, 0.0) node[above]{$a_1'$};
\draw    [shift={( 2, 0)}] (1.0, 0.0) node[above]{$a_2'$};
\draw    [shift={( 6, 0)}] (1.0, 0.0) node[above]{$a_{m_0-1}'$};
\draw[->][shift={(-1, 1)}] (0.2,-0.2) -- (0.8,-0.8);
\draw    [shift={(-1, 1)}] (0.5,-0.5) node[right]{$\alpha$};
\draw    [shift={(-1, 1)}] (0,0) node{$v_1=v$};
\draw    [shift={(-2, 0)}] (0,0) node{$v_2$};
\draw    [shift={(-5,-3)}] (0,0) node{$v_n$};
\draw[->][shift={(-1, 1)}] (-0.2,-0.2) -- (-0.8,-0.8);
\draw[->][shift={(-2, 0)}] (-0.2,-0.2) -- (-0.8,-0.8);
\draw    [shift={(-3,-1)}] [line width=1pt][dotted]
                           (-0.2,-0.2) -- (-0.8,-0.8);
\draw[->][shift={(-4,-2)}] (-0.2,-0.2) -- (-0.8,-0.8);
\end{tikzpicture}
\caption{The strings respectively corresponding to $P(v)$ and $P(w)$}
\label{fig:decomp-eA-1}
\end{figure}

Next, we show that $g$ is a homomorphism does not through any indecomposable projective module $P$ ($\not\cong P(w)$).
To do this, we assume that $g$ has a decomposition $g=g_1g_2$ such that the following diagram
\[\xymatrix{
P(w) \ar[dd]_{g_2} \ar[rr]^{h_{\alpha}} \ar[rd]_{g} & & P(v)\\
& \alpha A \ar[ru]_f & \\
P(u) \ar[ru]_{g_1} & &
}\]
commutes. Then there is a path $\wp$ on $(\Q,\I)$ whose source and sink respectively are $u$ and $w$,
such that $g_2: P(w)=\e_wA \to P(u)=\e_uA$, written as $h_{\wp}$,
sends each $\e_wa \in \e_wA$ to $\wp\cdot \e_w a = \e_u(\wp a)\in \e_u A$, i.e.,
\begin{align}\label{formula:decomp-eA}
  g_2(\e_wa) = \wp a
\end{align}
Notice that we have the following two facts.
\begin{itemize}
  \item[(1)]
    The module $\alpha A$ is a string module satisfying
    \[\M^{-1}(\alpha A) =\ w_1 \rarrow{a_1} w_2 \cdots \rarrow{a_{\tilde{m}-1}} w_{\tilde{m}}\
    (1\leqslant \tilde{m} \leqslant m), \]
    and the factor substring of $\M^{-1}(\alpha A)$ respect to $(0, a_1)$ is $\e_w$,
    see the mark (I) in \Pic \ref{fig:decomp-eA-2};
  \item[(2)]
    The string corresponded by $P(u)$ is of the form
    \begin{center}
      $\M^{-1}(P(u)) =\ \cdots \larrow{} u
      \xymatrix{ \ar@{~>}[r]^{\wp}& }
      w \rarrow{a_1'} w_2' \rarrow{} \cdots \rarrow{a_{\tilde{m}_0-1}'} w_{\tilde{m}_0}' $
    \end{center}
    where $0\leqslant \tilde{m}_0 \leqslant m_0$, and
    $\wp = c_1\cdots c_l$ ($c_1,\ldots, c_l \in \Q_1$) is a path such that,
    by the definition of string algebra and $\alpha a_1\notin\I$, we have
    \begin{align}\label{formula:decomp-eA 2}
      c_la_1\in\I
    \end{align}
    see the mark (II) in \Pic \ref{fig:decomp-eA-2}.
\end{itemize}
\begin{figure}[htbp]
\centering
\begin{tikzpicture}
\draw[blue!25][line width=20pt] (8.5,0) -- (0,0) -- (5.5,-5.5);
\draw[blue][->] (7.5,-.15) -- (7.5,-1.3);
\draw[blue] (7.5,-1.3) node[below]{$P(w)$};
\draw[green!50][line width=12pt] (-0.5,-3.5) -- (-2,-2) -- (0,0) -- (5,0);
\draw[green][->] (4,0) -- (4,-1.3);
\draw[green] (4,-1.3) node[below]{$P(u)$};
\draw[red!25][line width=12pt] (-5,-3) -- (-1,1) -- (3.5,-3.5);
\draw[red][->] (3,-2.7)--(4.8,-2.7);
\draw[red] (4.8,-2.7) node[right]{$P(v)$};
%
\draw[purple][rotate around={-45:(-2.5+4.5,0-1.5)}]
     [line width=1pt][shift={( 0.05,-0.35)}]
     (-2.5+4.5,0-1.5) ellipse (2.9cm and 0.7cm);
\draw[purple][->] (3.65,-3) -- (4.8,-3);
\draw[purple] (4.8,-3) node[right]{$\alpha A \mathop{\mapsto}\limits^{\M} a_1a_2\cdots a_{\tilde{m}}$};
\draw[purple] (4.9,-3) node[below]{(I)};
\draw (4.5,-0.5) node{(II)};
\draw[gray] (-2,-2) -- (3,-3);
\draw[gray] [shift={(-0.05,-0.05)}] (-2,-2) -- (3,-3);
\draw[gray] (0.5,-2.5) node[below]{(III) $u=w_i$};
\draw                      (0,0) node[left]{$w=\ $};
\draw    [shift={( 0, 0)}] (0,0) node{$w_1$};
\draw    [shift={( 1,-1)}] (0,0) node{$w_2$};
\draw    [shift={( 2,-2)}] (0,0) node{$w_{i-1}$};
\draw    [shift={( 3,-3)}] (0,0) node{$w_i$};
\draw    [shift={( 5,-5)}] (0,0) node{$w_m$};
\draw    [shift={( 2, 0)}] (0,0) node{$w_2'$};
\draw    [shift={( 8, 0)}] (0.15,0) node{$w_{m_0}'$};
\draw[->][shift={( 0, 0)}] (0.2,-0.2) -- (0.8,-0.8);
\draw    [shift={( 1,-1)}] [line width=1pt][dotted]
                           (0.2,-0.2) -- (0.8,-0.8);
\draw[->][shift={( 2,-2)}] (0.2,-0.2) -- (0.8,-0.8);
\draw    [shift={( 3,-3)}] [line width=1pt][dotted]
                           (0.2,-0.2) -- (0.8,-0.8);
\draw[->][shift={( 4,-4)}] (0.2,-0.2) -- (0.8,-0.8);
\draw    [shift={( 0, 0)}] (0.5,-0.5) node[right]{$a_1$};
\draw    [shift={( 2,-2)}] (0.5,-0.5) node[right]{$a_{i-1}$};
\draw    [shift={( 4,-4)}] (0.5,-0.5) node[right]{$a_{m-1}$};
\draw[->][shift={( 0, 0)}] (0.3, 0.0) -- (1.7, 0.0);
\draw[->][shift={( 2, 0)}] (0.3, 0.0) -- (1.7, 0.0);
\draw    [shift={( 4, 0)}] [line width=1pt][dotted]
                           (0.3, 0.0) -- (1.7, 0.0);
\draw[->][shift={( 6, 0)}] (0.3, 0.0) -- (1.7, 0.0);
\draw    [shift={( 0, 0)}] (1.0, 0.0) node[above]{$a_1'$};
\draw    [shift={( 2, 0)}] (1.0, 0.0) node[above]{$a_2'$};
\draw    [shift={( 6, 0)}] (1.0, 0.0) node[above]{$a_{m_0-1}'$};
\draw[->][shift={(-1, 1)}] (0.2,-0.2) -- (0.8,-0.8);
\draw    [shift={(-1, 1)}] (0.5,-0.5) node[right]{$\alpha$};
\draw    [shift={(-1, 1)}] (0,0) node{$v_1=v$};
\draw    [shift={(-2, 0)}] (0,0) node{$v_2$};
\draw    [shift={(-5,-3)}] (0,0) node{$v_n$};
\draw[->][shift={(-1, 1)}] (-0.2,-0.2) -- (-0.8,-0.8);
\draw[->][shift={(-2, 0)}] (-0.2,-0.2) -- (-0.8,-0.8);
\draw    [shift={(-3,-1)}] [line width=1pt][dotted]
                           (-0.2,-0.2) -- (-0.8,-0.8);
\draw[->][shift={(-4,-2)}] (-0.2,-0.2) -- (-0.8,-0.8);
\draw (-2,-2) node{$u$};
\draw [->][domain=-2.0:0,samples=160,smooth]
      [line width=2pt][rotate around={45:(-1,0)}][shift={(-0.8,-0.8)}]
      plot (\x,{0.05*cos(\x*2*pi*150)});
\draw (-1,-1) node[below right]{$\wp$};
\draw[->][shift={( 0, 0)}] (-1.8,-2.2) -- (-1.2,-2.8);
\draw    [shift={( 1,-1)}]  [line width=1pt][dotted]
                           (-1.8,-2.2) -- (-1.2,-2.8);
\end{tikzpicture}
\caption{If $g$ can be decomposed through $P(u)$ ($\not\cong P(w)$)}
\label{fig:decomp-eA-2}
\end{figure}
Now we consider the canonical decomposition of $g_1: P(u) \to \alpha A$,
it is easy to see that the image $\image(g_1)$ is a string module
whose string is a factor substring of $\M^{-1}(P(u))$ which is of the form
\begin{center}
  $ \ \cdots \larrow{} u
    \xymatrix{ \ar@{~>}[r]^{\wp}& }
   w \rarrow{a_1'} w_2' \rarrow{} \cdots \rarrow{a_{\tilde{m}_1-1}'} w_{\tilde{m}_1}' $,
\end{center}
where $0\le \tilde{m}_1\le \tilde{m}_0$ $(\le m_0)$, see \Pic \ref{fig:decomp-eA-2}.
We obtain two cases as follows.
\begin{itemize}
  \item[(A)] The vertex $u$ is not a vertex on the string $\M^{-1}(\alpha A)$;
  \item[(B)] The vertex $u$ is a vertex on the string $\M^{-1}(\alpha A)$.
\end{itemize}

In case (A), we obviously have $g_1=0$ by Theorem \ref{thm:Kra1991} (1). This is a contradiction.

Now we show that (B) admits a contradiction and end my proof.
In ths case, we obtain $u = w_i$ for some $1\le i\le \tilde{m}$.
It follows that $a_{i-1}$ is an arrow ending with $u$,
see the mark (III) in \Pic \ref{fig:decomp-eA-2}.
One can check that $\M^{-1}(P(u))$ has a factor substring coincides with
an image substring of $\M^{-1}(\alpha A)$, it describe the homomorphism
\begin{center}
  $g_1: P(u)=\e_uA \to \alpha A$,  $\e_u a \mapsto \alpha a_1\cdots a_{i-1} \e_{w_i} a$ ($\forall a\in A$)
\end{center}
which is non-zero.
However, $g_2$ sends each element $\e_wa$ ($\forall a\in A$) in $P(w)$ to
the element $\wp a$ in $P(v)$, see (\ref{formula:decomp-eA}). Thus,
\begin{center}
  $g(\e_w a) = g_1g_2(\e_w a) = g_1(\wp a) = \alpha a_1\cdots a_{i-1} \wp a $
\end{center}
It follows that
\begin{align*}
    fg(a_1\cdots a_{i-1})
  & = f(\alpha a_1\cdots a_{i-1} \wp a_1\cdots a_{i-1}) \\
  & = \alpha a_1\cdots a_{i-1} \wp a_1\cdots a_{i-1} = 0\ (\text{by } (\ref{formula:decomp-eA 2}))
\end{align*}
By using $h_{\alpha}=fg$, we have
\begin{align*}
      fg(a_1\cdots a_{i-1})
  & = h_{\alpha}(a_1\cdots a_{i-1}) \\
  & = \alpha a_1\cdots a_{i-1} \ne 0.
\end{align*}
We obtain a contradiction.

We can show that $f$ can not be decomposed through any indecomposable projective module by similar way.
\end{proof}

\begin{lemma}
\label{lemm:decomp-alphaA}
Keep the notations from Lemma \ref{lemm:decomp}.
For any $\beta\in\Q_1$, the homomorphisms $f$ and $g$ can not be decomposed through $\beta A$ $(\not\cong \beta A)$.
\end{lemma}

The proof of the above proposition is similar to that of Lemma \ref{lemm:decomp-alphaA}.
In the proof of Lemma \ref{lemm:decomp-alphaA},
we prove that $g$ can not be decomposed through any indecomposable projective module $P(u)$ ($\not\cong P(w)$).
Now, we show that $f$ can not be decomposed through any $\beta A$ ($\not\cong \alpha A$).

\begin{proof}
Similar to the proof of Lemma \ref{lemm:decomp-eA}, assume
\[ \M^{-1}(P(w)) =\ w_{m_0}' \larrow{a_{m_0-1}'} \cdots \larrow{a_2'} w_2' \larrow{a_1'}
  w_1 \rarrow{a_1} w_2 \rarrow{a_2} \cdots \rarrow{} w_{m-1} \rarrow{a_{m-1}} w_m \]
and
\[ \M^{-1}(P(v)) =\ v_n \larrow{b_{n-1}} \cdots \larrow{b_2} v_2 \larrow{b_1} v_1
   \rarrow{\alpha} w_1 \rarrow{a_1} \cdots \rarrow{a_{m'-1}} w_{m'} \]
($1 \leqslant m' \leqslant m$) which are shown in  \Pic \ref{fig:decomp-eA-1}.

Next, we show that $f$ is a homomorphism does not through any $\beta A$ ($\forall \beta\in \Q_1$, $\beta\ne \alpha$).
To do this, we assume that $f$ has a decomposition $f=f_1f_2$ such that the following diagram
\[\xymatrix{
P(w) \ar[rr]^{h_{\alpha}} \ar[rd]_{g} & & P(v)\\
& \alpha A \ar[ru]_f \ar[rd]_{f_2} & \\
& & \beta A \ar[uu]_{f_1}
}\]
commutes.
Notice that the string $\M^{-1}(\beta A)$ is of the form
\begin{center}
  $ u_1 \rarrow{c_1} u_2 \rarrow{c_2} \cdots \rarrow{c_{l-1}} u_l$,
\end{center}
then any image substring of it is of the form $c_{l'}c_{l'+1}\cdots c_{l-1}$
($1\leqslant l'\leqslant l-1$, $\t(\beta)=\s(c_{l'})$,
in the case of $l'=l-1$ we take image substring is $\e_{u_l}$),
and, by Theorem \ref{thm:Kra1991} (1) and $f_2 \ne 0$, there is a factor substring of
$\M^{-1}(\alpha A) = a_1a_2\cdots a_{\tilde{m}}$,
written as $a_1a_2\cdots a_j$ ($0 \leqslant j \leqslant \tilde{m}$), coincides with
some image substring of $\M^{-1}(\beta A)$.
It follows that
\begin{center}
  $c_{l'} = a_1$, $c_{l'+1} = a_2$, $\ldots$, $c_l = a_j$

  hold for some $0\leqslant j\leqslant \tilde{m}$, where $j=l-l'+1$

  (see mark ``Case (A)'' in \Pic \ref{fig:decomp-alphaA}).
\end{center}
\begin{figure}[htbp]
\centering
\begin{tikzpicture}
\draw[blue!25][line width=20pt] (8.5,0) -- (0,0) -- (5.5,-5.5);
\draw[blue][->] (7.5,-.15) -- (7.5,-1.3);
\draw[blue] (7.5,-1.3) node[below]{$P(w)$};
\draw[red!75][line width=12pt] (-5,-3) -- (-1,1) -- (3.5,-3.5);
\draw[red!75][->] (-1,1)--(2,1);
\draw[red!75] (2,1) node[right]{$P(v)$};
\draw[purple][rotate around={-45:(-2.5+4.5,0-1.5)}]
     [line width=1pt][shift={( 0.05,-0.35)}]
     (-2.5+4.5,0-1.5) ellipse (2.9cm and 0.7cm);
\draw[purple][->] (3.65,-3) -- (4.8,-3);
\draw[purple] (4.8,-3) node[right]{$\alpha A \mathop{\mapsto}\limits^{\M} a_1a_2\cdots a_{\tilde{m}}$};
\draw[cyan][line width=14pt] (-5.4,-5.4) --(0,0);
\draw[cyan][->] (-5.4,-5.4) -- (-4.4,-5.4);
\draw[cyan](-4.4,-5.4) node[right]{$\beta A$};
\draw[cyan](-4.4,-5.8) node[right]{Case (B)};
\draw[yellow][line width=8pt] (-5.2,-5.2) --(0,0) -- ( 3.2,-3.2);
\draw[yellow][line width=0.85pt][->] (1.7,-1.5) -- (4.6,-1.5);
\draw[yellow] (4.8,-1.5) node[right]{$\pmb{\beta A}$};
\draw[yellow] (4.8,-1.8) node[right]{Case (A)};
\draw[red] (3.5,-3.5) arc(-45:-225:2.45);
\draw[red] (3.45,-3.45) arc(-45:-225:2.375);
\draw[red] (2,-4.2) node[below]{$w_{m'}=w_1$};
\draw                      (0,0) node[left]{$w=\ $};
\draw    [shift={( 0, 0)}] (0,0) node{$w_1$};
\draw    [shift={( 1,-1)}] (0,0) node{$w_2$};
\draw    [shift={( 2,-2)}] (0,0) node{$w_j$};
\draw    [shift={( 3,-3)}] (0,0) node{$w_{j+1}$};
\draw    [shift={( 5,-5)}] (0,0) node{$w_m$};
\draw    [shift={( 2, 0)}] (0,0) node{$w_2'$};
\draw    [shift={( 8, 0)}] (0.15,0) node{$w_{m_0}'$};
\draw[->][shift={( 0, 0)}] (0.2,-0.2) -- (0.8,-0.8);
\draw    [shift={( 1,-1)}] [line width=1pt][dotted]
                           (0.2,-0.2) -- (0.8,-0.8);
\draw[->][shift={( 2,-2)}] (0.2,-0.2) -- (0.8,-0.8);
\draw    [shift={( 3,-3)}] [line width=1pt][dotted]
                           (0.2,-0.2) -- (0.8,-0.8);
\draw[->][shift={( 4,-4)}] (0.2,-0.2) -- (0.8,-0.8);
\draw    [shift={( 0, 0)}] (0.5,-0.5) node[right]{$a_1$};
\draw    [shift={( 2,-2)}] (0.5,-0.5) node[right]{$a_j$};
\draw    [shift={( 4,-4)}] (0.5,-0.5) node[right]{$a_{m-1}$};
\draw[->][shift={( 0, 0)}] (0.3, 0.0) -- (1.7, 0.0);
\draw[->][shift={( 2, 0)}] (0.3, 0.0) -- (1.7, 0.0);
\draw    [shift={( 4, 0)}] [line width=1pt][dotted]
                           (0.3, 0.0) -- (1.7, 0.0);
\draw[->][shift={( 6, 0)}] (0.3, 0.0) -- (1.7, 0.0);
\draw    [shift={( 0, 0)}] (1.0, 0.0) node[above]{$a_1'$};
\draw    [shift={( 2, 0)}] (1.0, 0.0) node[above]{$a_2'$};
\draw    [shift={( 6, 0)}] (1.0, 0.0) node[above]{$a_{m_0-1}'$};
\draw[->][shift={(-1, 1)}] (0.2,-0.2) -- (0.8,-0.8);
\draw    [shift={(-1, 1)}] (0.5,-0.5) node[right]{$\alpha$};
\draw    [shift={(-1, 1)}] (0,0) node{$v_1=v$};
\draw    [shift={(-2, 0)}] (0,0) node{$v_2$};
\draw    [shift={(-5,-3)}] (0,0) node{$v_n$};
\draw[->][shift={(-1, 1)}] (-0.2,-0.2) -- (-0.8,-0.8);
\draw[->][shift={(-2, 0)}] (-0.2,-0.2) -- (-0.8,-0.8);
\draw    [shift={(-3,-1)}] [line width=1pt][dotted]
                           (-0.2,-0.2) -- (-0.8,-0.8);
\draw[->][shift={(-4,-2)}] (-0.2,-0.2) -- (-0.8,-0.8);
\draw[->][shift={( 0, 0)}] (-0.8,-0.8) -- (-0.2,-0.2);
\draw[->][shift={(-1,-1)}] (-0.8,-0.8) -- (-0.2,-0.2);
\draw    [shift={(-2,-2)}] [line width=1pt][dotted]
                           (-0.8,-0.8) -- (-0.2,-0.2);
\draw[->][shift={(-3,-3)}] (-0.8,-0.8) -- (-0.2,-0.2);
\draw[->][shift={(-4,-4)}] (-0.8,-0.8) -- (-0.2,-0.2);
\draw[shift={(-1,-1)}] (0,0) node{$u_{l'-1}$};
\draw[shift={(-2,-2)}] (0,0) node{$u_{l'-2}$};
\draw[shift={(-3,-3)}] (0,0) node{$u_2$};
\draw[shift={(-4,-4)}] (0,0) node{$u_1$};
\draw[shift={(-5,-5)}] (0,0) node{$u$};
\draw[shift={( 0, 0)}] (-0.4,-0.4) node[left]{$c_{l'-1}$};
\draw[shift={(-1,-1)}] (-0.4,-0.4) node[left]{$c_{l'-2}$};
\draw[shift={(-3,-3)}] (-0.4,-0.4) node[left]{$c_1$};
\draw[shift={(-4,-4)}] (-0.4,-0.4) node[left]{$\beta$};
\end{tikzpicture}
\caption{If $f$ can be decomposed through $\beta A$ ($\not\cong \alpha A$)}
\label{fig:decomp-alphaA}
\end{figure}
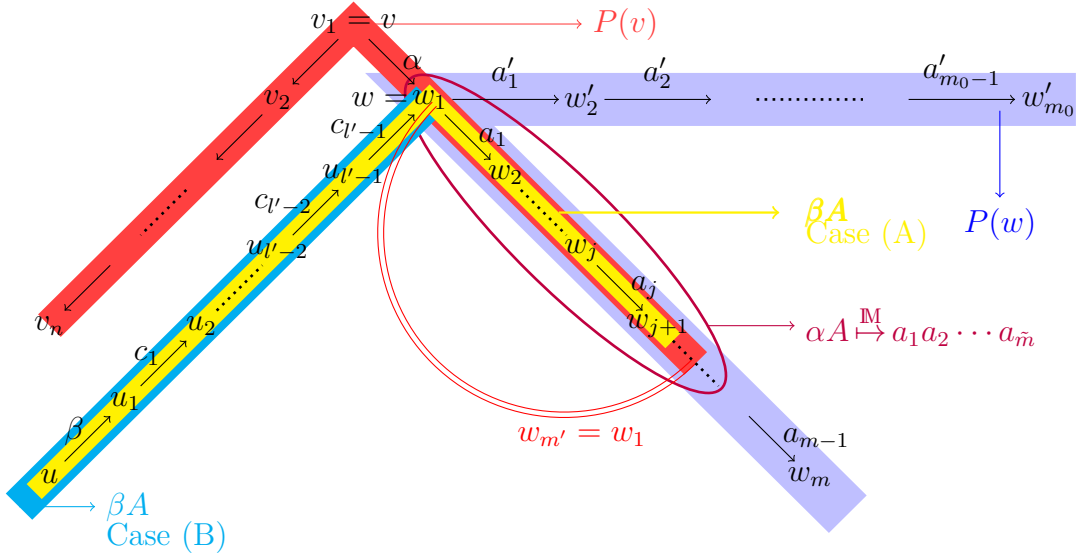
We obtain two cases as follows.
\begin{itemize}
  \item[(A)] $l'<l-1$, that is, the length of $c_{l'}c_{l'+1}\cdots c_{l-1}$ is great than or equal to $1$,
    see mark ``Case (A)'' in \Pic \ref{fig:decomp-alphaA}.

  \item[(B)] $l'=l-1$, that is, $\M(c_{l'}c_{l'+1}\cdots c_{l-1}) = \M(\e_{u_{l}})$ is isomorphic to the simple module $S(u_l)$, see mark ``Case (B)'' in \Pic \ref{fig:decomp-alphaA}.
\end{itemize}

In Case (A), $\alpha a_1\notin \I$ admits $c_{l'-1}a_1\in\I$ by the definition of string algebra,
i.e., $c_{l'-1}c_{l'}\in \I$, this is a contradiction since $c_1c_2\cdots c_{l-1}$ is a string.

In Case (B), we have $\image(f_2) \cong S(w_1)$
is a simple module which must be a submodule of $P(v)$
since $f_1 : \beta A \to P(v)$ is non-zero.
It follows that $w_{m'} = w_1$ by using Theorem \ref{thm:Kra1991} (1),
then, by using the definition of string algebra, we have
\begin{itemize}
  \item[(B.1)] $a_{m'-1}$ coincides with the arrow $\alpha$;
  \item[(B.2)] or $a_{m'-1}$ coincides with the arrow $c_{l'-1}$ ($=c_{l}$).
\end{itemize}
On the other hand, by $f_2\ne 0$, there is a factor substring of $\M^{-1}(\beta A)$
$=c_1\cdots c_{l'-1}$, say
\begin{center}
  $c_1\cdots c_{l''}$ $(0 \leqslant l'' \leqslant l'-1)$,
\end{center}
coincides with some image substring
\begin{center}
  $a_{m''}\cdots a_{m'-1}$ $(0 \leqslant m'' \leqslant m'-1)$
\end{center}
of $\M^{-1}(P(v))$, i.e.,
\begin{align}\label{formula:decomp-alphaA 1}
  c_1\cdots c_{l''} = a_{m''}\cdots a_{m'-1}\ (l'=m'-m'').
\end{align}
\begin{itemize}
\item[In](B.1):
We get $f_2$ is of the following form
\begin{align}\label{formula:decomp-alphaA 2}
  f_2: \alpha A \to \beta A,\ \alpha a\mapsto  \beta a_{m''} \cdots a_{m'-2} \cdot \alpha a,
\end{align}
by using (\ref{formula:decomp-alphaA 1}),
and get $f_1$ is of the following form
\begin{align}\label{formula:decomp-alphaA 3}
  f_1: \beta A \to \alpha A,\ \beta a \mapsto \alpha a_1\cdots a_{m''-1} a.
\end{align}
Consider the path $a_1$ as an element in $A$, we have
\begin{align}
    h_{\alpha}(a_1)
& = f_1(f_2(g(a_1))) = f_1(f_2(\alpha a_1)) \nonumber \\
& \mathop{=}\limits^{(\ref{formula:decomp-alphaA 2})}
  f_1(\beta a_{m''} \cdots a_{m'-2}\alpha a_1) \nonumber  \\
& \mathop{=}\limits^{(\ref{formula:decomp-alphaA 3})}
  \alpha a_1\cdots a_{m''-1}a_{m''} \cdots a_{m'-2}\alpha a_1 \nonumber \\
& = \alpha a_1\cdots a_{m'-2} \alpha a_1 \nonumber \\
& \ne \alpha a_1 = h_{\alpha}(a_1), \nonumber
\end{align}
a contradiction.

\item[In](B.2):
  $a_{m'-1} = c_{l'-1}$ admits $a_{m'-2}c_{l'-1} \in \I$ by using the definition of string algebra.
  Then $a_{m'-2}c_{l'-1} = a_{m'-2}a_{m'-1}$, as an element in $A$, is zero.
  It contradicts with $a_1\cdots a_{m'-2}a_{m'-1}$ is a string.
\end{itemize}

The contradictions given by Cases (A) and (B) show that this proposition holds.
\end{proof}

\begin{lemma} \label{lemm:decomp-two proj}
Let $A = \kk\Q/\I$ be a string algebra and $\alpha$ be an arrow on $\Q$.
Any homomorphism $h_{\alpha}: P(v) \to P(w)$ induced by the arrow $\alpha: w\to v$
can not be decomposed through arbitrary $\beta A$ $(\not\cong \alpha A)$.
\end{lemma}

\begin{proof}
It is well-known that each base $h_p$ element of $\Hom_A(P(v),P(w))$ is described by the path $p$ from $w$ to $v$ on $(\Q,\I)$, that is, $h_p: e_v a\mapsto p\cdot e_va = pa$.
If $h_p$ can be decomposed through $\beta A$, then, by Theorem \ref{thm:Kra1991},
one can check that $\beta$ is an arrow such that:
\begin{itemize}
  \item[(1)] $\t(\beta) = \s(\M^{-1}(P(v)))$,
  \item[(2)] $p$ is a factor substring of $M^{-1}(P(w))$,
  \item[(3)] $\beta$ is an arrow on $p$.
\end{itemize}
If $p=\alpha$ is an arrow, then $\beta$ must be coincided with $\alpha$,
it contradicts with $\alpha A\not\cong \beta A$ as required.
\end{proof}

\begin{lemma} \label{lemm:decomp-two alphaA}
Let $A = \kk\Q/\I$ be an SAG-algebra, $\alpha$ be a left forbidden arrow,
and $\beta$ be an arbitrary arrow. If $\Hom_A(\alpha A, \beta A) \ne 0$,
then it can be decomposed through some indecomposable projective module.
\end{lemma}

\begin{proof}
Assume $f$ is a non-zero homomorphism which can be decomposed through $\gamma A$,
here, $\gamma A \not\cong \alpha A$ and $\gamma A \not\cong \beta A$.
Assume $\alpha$ is an arrow on the path
\begin{center}
  $p=$ \ $\xymatrix{
    \cdots \ar[r] & \bullet \ar[r]^{a_1} & \bullet \ar[r]^{a_2} & \cdots \ar[r]^{a_n}
  & v \ar[r]^{\alpha} & w \ar@{~>}[r]^{q_{\alpha}} & \bullet
  }$
\end{center}
such that $a_1a_2, \ldots a_{n-1}a_n, a_n\alpha\in \I$ and $\M(q_{\alpha}) \cong \alpha A$.
By Theorem \ref{thm:Kra1991} and $f \ne 0$, we obtain that
$q_{\alpha}$ has a factor substring which coincides with an image substring of $q_{\beta}:=\M^{-1}(\beta A)$
(then $\t(\beta) = \s(q_{\beta})$ holds).
Thus, $\t(q_{\beta})$ is a vertex on $\alpha q_{\alpha}$,
and the positional relationship of $q_{\alpha}$ and $q_{\beta}$ is one of the following forms:
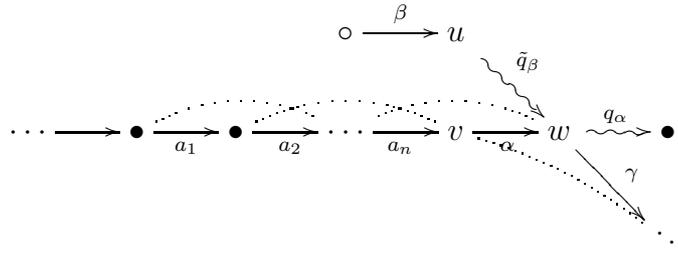
\begin{itemize}
\item[C]\hspace{-6pt}ase 1.
\begin{center}
    $\xymatrix{
   & & & \circ \ar[r]^{\beta} & u \ar@{~>}[rd]^{\tilde{q}_{\beta}} & & \\
    \cdots \ar[r] & \bullet \ar[r]_{a_1} \ar@{.}@/^1pc/[rr]
   & \bullet \ar[r]_{a_2} \ar@{.}@/^1pc/[rr] & \cdots \ar[r]_{a_n} \ar@{.}@/^1pc/[rr]
   & v \ar[r]_{\alpha} \ar@{.}@/^0.5pc/[rrd] & w \ar@{~>}[r]^{q_{\alpha}} \ar[rd]^{\gamma} & \bullet; \\
   & & & & & & \ddots &
  }$
\end{center}
where $q_{\beta}$ is of the form $\tilde{q}_{\beta}q'_{\alpha}$,
and $q'_{\alpha}$ is a factor image of $q_{\alpha}$.

\item[C]\hspace{-6pt}ase 1.
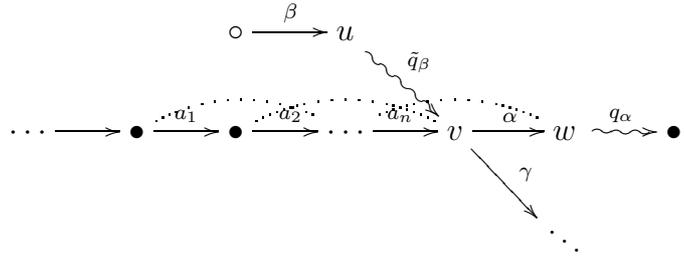
\begin{center}
$\xymatrix{
   & & \circ \ar[r]^{\beta} & u \ar@{~>}[rd]^{\tilde{q}_{\beta}} & & & \\
    \cdots \ar[r] & \bullet \ar[r]^{a_1} \ar@{.}@/^1pc/[rr]
   & \bullet \ar[r]^{a_2} \ar@{.}@/^1pc/[rr] & \cdots \ar[r]^{a_n} \ar@{.}@/^1pc/[rr]
   & v \ar[r]^{\alpha}  \ar[rd]^{\gamma} & w \ar@{~>}[r]^{q_{\alpha}} & \bullet, \\
   & & & & & \ddots & &
}$
\end{center}
where $q_{\beta}$ is of the form $\tilde{q}_{\beta}\alpha q'_{\alpha}$,
and $q'_{\alpha}$ is a factor image of $q_{\alpha}$.
\end{itemize}

In Case 1, we assume
\begin{align} \label{formula in lemm:decomp-two alphaA}
  q_{\alpha} =\ \xymatrix{ w \ar[r]^{q_{\alpha,1}} & \bullet \ar[r]^{q_{\alpha,2}} & \cdots \ar[r]^{q_{\alpha,l}}& \bullet};
  \text{ and }
  \tilde{q}_{\beta} =\ \xymatrix{ u \ar[r]^{q_{\beta,1}} & \circ \ar[r]^{q_{\beta,2}} & \cdots \ar[r]^{q_{\beta,\ell}} & w };
  \nonumber
\end{align}
then $q_{\beta,\ell}q_{\alpha,1} \in \I$  since $A$ is a string algebra.
We know that $q_{\alpha}$ has a factor substring, say $r$, coinciding with an image substring of $q_{\beta}$,
there are two subcases as follows.
\begin{itemize}
\item[S]\hspace{-6pt}ubcase 2.1. $l=0$.
Then, $\tilde{q}_{\beta} = q_{\beta}$, and $r = q_{\alpha}' = q_{\alpha} = \e_w$ is both a factor substring of $q_{\alpha}$ and an image substring of $q_{\beta}$ respect to $q_{\beta,\ell}$.
It follows that $\t(q_{\beta}) = w$, and $\image(f) \cong S(w)$ is simple.
On the other hand, $A$ is an SAG-algebra, then all relations in $\I$ are paths of length two,
and so, $\beta q_{\beta}\gamma = 0$ admits that $q_{\beta,\ell}\gamma \in \I$.
Thus, $S(w)$ is a direct summand of the socle of $P(S(w)) = P(\s(q_{\beta, 1})) = P(u)$.
That is, for any $a\in A$, we obtain a decomposition
\[ \xymatrix@C=1.5cm{ \alpha A \ar@/^1pc/[rr]^{f} \ar[r]_{f_1} & P(u) \ar[r]_{f_2} & \beta A } \]
of $f$ such that $f(\alpha a) = \beta q_{\beta,1}\cdots q_{\beta,\ell-1} \cdot f_1(\alpha a)$,
where $f_1$ sends $\alpha a$ to an element in $P(u)$ which is of the form $\e_{u} a'$,
and, for any $x\in A$, $f_2$ sends each $\e_{u} x$ to $\beta q_{\beta,1}\cdots q_{\beta,l-1} \cdot \e_{u} x$.

\item[S]\hspace{-6pt}ubcase 2.2. $l\ge 1$.
We can show that $f$ can be decomposed by $P(u)$ by the method similar to the proof of Subcase 2.1.
\end{itemize}

In Case 2, we assume
\begin{align}
  q_{\alpha} =\ \xymatrix{ w \ar[r]^{q_{\alpha,1}} & \bullet \ar[r]^{q_{\alpha,2}} & \cdots \ar[r]^{q_{\alpha,l}}& \bullet};
  \text{ and }
  \tilde{q}_{\beta} =\
  \xymatrix{ u \ar[r]^{q_{\beta,1}}
  & \circ \ar[r]^{q_{\beta,2}}
  & \cdots \ar[r]^{q_{\beta,\ell}} & v };
  \nonumber
\end{align}
then $q_{\beta,\ell}\alpha\notin\I$, which admits $q_{\beta,\ell}\gamma\in\I$ by using $A$ to be a string algebra.
Assume $q'_{\alpha} = q_{\alpha,1}\cdots q_{\alpha,l'}$, $1\leqslant l\leqslant l'$,
then, for any $a\in A$, the homomorphism $f$ described by $q'_{\alpha}$ sends each element $\alpha a$ to $\beta\tilde{q}_{\beta}\alpha a$ ($\in \beta A$).
Notice that
\begin{center}
  $\tilde{q}_{\beta}\alpha a = \e_{\s(\tilde{q}_{\beta,1})}\tilde{q}_{\beta}\alpha a
  = \e_u\tilde{q}_{\beta}\alpha a \in \e_uA = P(u)$,
\end{center}
now one can check that $f$ can be decomposed through $P(u)$.
\end{proof}

\section{$\R$-endomorphism algebras}

For a string algebra $A$, let $\R$ be a subset of $\Q_1$ whose all elements are left forbidden arrows in this section.

\begin{definition} \label{def:R} \rm
An {\defines $\R$-summed module} is the direct sum of $A$ and all arrowed modules $\alpha A$ ($\alpha \in \R$), that is,
\[ M_{\R} := A \oplus \bigoplus_{\alpha\in\R} \alpha A, \]
and its endomorphism algebra
\[ A_{\R} := \End_A(M_{\R}) \]
is called an {\defines $\R$-endomorphism algebra}.
The set $\R$ is called an {\defines left forbidden arrow index} (of $(\Q,\I)$).
In particular, there are two remarks as follows.
\begin{itemize}
  \item[(1)] In the case of $\R = \varnothing$, the $\R$-summed module is $0$.
  \item[(2)] Every arrowed module $\alpha A$ is an $\R$-summed module with $\R=\{\alpha\}$.
\end{itemize}
\end{definition}

In this section, we provide a method to compute the bound quiver of $A_{\R}$ in the case of $A$ to be an SAG-algebra.

\subsection{$\R$-bound quiver} \label{subsect:R bound quiver}

Let $(\Q,\I)$ be a bound quiver of a string algebra.
We define its $\R$-bound quiver is $(\R(\Q),\R(\I))$,
where $\R(\Q) = (\R(\Q)_0, \R(\Q)_1, {\R}(\s), {\R}(\t))$ is given by the following Steps 1--4,
and $\R(\I)$ is given by the following Steps 5--6.
\begin{itemize}
  \item[S]\hspace{-6pt}tep 1  $\R(\Q)_0 := \Q_0 \cup \Q_0^{\spl}$, where $\Q_0^{\spl}$ is a finite set such that
    the bijection
    \begin{center}
      $\mathfrak{v}: \{\alpha A \mid \alpha\in \R\} \to \Q_0^{\spl}$
    \end{center}
    exist.

  \item[S]\hspace{-6pt}tep 2 $\R(\Q)_1 := (\Q_1\backslash\R)\cup\R^{\spl}$,
    where $\R^{\spl} := \{ \alpha_{\Left}, \alpha_{\Right} \mid \alpha \in \R \}$.

  \item[S]\hspace{-6pt}tep 3 ${\R}(\s): \R(\Q)_1 \to \R(\Q)_0$ sends any arrow $a \in \Q_1\backslash\R$ to its source $\s(a)$,
    sends any arrow $\alpha_{\Left} \in \R^{\spl}$ to the source $\s(\alpha)$ of $\alpha$,
    and sends any arrow $\alpha_{\Right} \in \R^{\spl}$ to the vertex $\mathfrak{v}(\alpha A)$.

  \item[S]\hspace{-6pt}tep 4 ${\R}(\t): \R(\Q)_1 \to \R(\Q)_0$ sends any arrow $a \in \Q_1\backslash\R$ to its sink $\t(a)$,
    sends any arrow $\alpha_{\Left} \in \R^{\spl}$ to the vertex $\mathfrak{v}(\alpha A)$,
    and sends any arrow $\alpha_{\Right} \in \R^{\spl}$ to the sink $\t(\alpha)$ of $\alpha$.

  \item[S]\hspace{-6pt}tep 5 For any arrow $a\in \Q_1$, define
  \begin{align*}
    & a^{\spl} =
    {\begin{cases}
      a, & \text{if } a\in \Q_1\backslash\R; \\
      a_{\Left}a_{\Right}, & \text{if } a\in \R,
    \end{cases}} \\
    & a^{\spl}_{\mathrm{L}} =
    {\begin{cases}
      a, & \text{if } a\in \Q_1\backslash\R; \\
      a_{\Left}, & \text{if } a\in \R,
    \end{cases}} \\
    \text{ and }
    & a^{\spl}_{\mathrm{R}} =
    {\begin{cases}
      a, & \text{if } a\in \Q_1\backslash\R; \\
      a_{\Right}, & \text{if } a\in \R.
    \end{cases}}
  \end{align*}
  For any path $p=a_1a_2\cdots a_n$ on $(\Q,\I)$, we define
  \begin{center}
    $p^{\spl} = (a_1)^{\spl}_{\Right}a_2^{\spl}\cdots a_{n-1}^{\spl} (a_n)^{\spl}_{\Left}$.
  \end{center}

  \item[S]\hspace{-6pt}tep 6
    $\R(\I) := \langle p^{\spl} \mid p\in \I \rangle$ which is naturally induced by $\I$ and Step 5.
\end{itemize}

\begin{remark}
\label{rmk:R(A) is string} \rm \
\begin{itemize}
  \item The finite-dimensional algebra given by $(\R(\Q), \R(\I))$ is written as $\R(A)$.
    If $A$ is a string algebra (resp., an SAG-algebra), then so is $\R(A)$.
  \item It is clear that $\alpha_{\Left}\alpha_{\Right} \notin \I(A)$.
  \item If $\R=\varnothing$, then it is trivial that $\R(A)$ and $A$ coincide.
\end{itemize}
\end{remark}

\begin{example} \label{examp:split} \rm
Let $A = \kk\Q/\I$ be the string algebra given by Example \ref{examp:string}.
Take $\R = \{a, d, a'\}$, then the $\R$-bound quiver $(\R(\Q),\R(\I))$ is shown in \Pic \ref{fig in examp:split}, where
\begin{align*}
  \R(\Q)_0
& = \{1,2,3,4,5,6,aA,dA,a'A\} \\
& = \Q_0\cup\{v_a=\mathfrak{v}(aA), v_{d}=\mathfrak{v}(dA), v_{a'}=\mathfrak{v}(a'A)\}; \\
  \R(\Q)_1
& = (\Q_1\backslash\R) \cup \{ a_{\Left}, d_{\Left}, a'_{\Left},
    a_{\Right}, d_{\Right}, a'_{\Right} \};
\end{align*}
\begin{figure}[htbp]
\centering
\begin{tikzpicture}[scale=2]
\draw[   red][dotted][line width = 0.8pt][rotate around={  0:(0,0)}]
  ( 0.51, 0.78) to[out=-60,in=60] ( 0.51,-0.78);
\draw[yellow][dotted][line width = 0.8pt][rotate around={120:(0,0)}]
  ( 0.51, 0.78) to[out=-60,in=60] ( 0.51,-0.78);
\draw[  blue][dotted][line width = 0.8pt][rotate around={240:(0,0)}]
  ( 0.51, 0.78) to[out=-60,in=60] ( 0.51,-0.78);
\draw[   red][dotted][line width = 0.8pt][rotate around={  0:(0,0)}]
  ( 1.50,-0.15) -- ( 1.00, 1.00);
\draw[yellow][dotted][line width = 0.8pt][rotate around={120:(0,0)}]
  ( 1.50,-0.15) -- ( 1.00, 1.00);
\draw[  blue][dotted][line width = 0.8pt][rotate around={240:(0,0)}]
  ( 1.50,-0.15) -- ( 1.00, 1.00);
\draw[   red][dashed][line width = 1.2pt][rotate around={  0:(0,0)}]
  ( 2.15,-0.15) arc (180:225:0.55);
\draw[yellow][dashed][line width = 1.2pt][rotate around={120:(0,0)}]
  ( 2.15,-0.15) arc (180:225:0.55);
\draw[  blue][dashed][line width = 1.2pt][rotate around={240:(0,0)}]
  ( 2.15,-0.15) arc (180:225:0.55);
\draw[   red][dotted][line width = 1.6pt][rotate around={  0:(0,0)}]
  (1.3,2.25) arc(60:115:2.6) -- (-0.25, 0.83) arc(120:215:1.1);
\draw[yellow][dotted][line width = 2.6pt][rotate around={120:(0,0)}]
  (2.6,0) arc(0:115:2.6) -- (-0.25, 0.83) arc(120:215:1.1);
\draw[blue!55][dotted][line width = 3.6pt][rotate around={240:(0,0)}]
  (2.6,0) arc(0:115:2.6) -- (-0.25, 0.83) arc(120:170:1.1);
\draw[   red][dotted][line width = 0.8pt][rotate around={  0:(0,0)}]
  (2.70,-0.15-0.60) arc( -90:  90:0.60);
\draw[   red][dashed][line width = 0.8pt][rotate around={  0:(0,0)}]
  (2.70,-0.15+0.75) arc(  90:-136:0.75);
\draw[yellow][dotted][line width = 0.8pt][rotate around={120:(0,0)}]
  (2.70,-0.15-0.60) arc( -90:  90:0.60);
\draw[yellow][dashed][line width = 0.8pt][rotate around={120:(0,0)}]
  (2.70,-0.15+0.75) arc(  90:-136:0.75);
\draw[  blue][dotted][line width = 0.8pt][rotate around={240:(0,0)}]
  (2.70,-0.15-0.60) arc( -90:  90:0.60);
\draw[  blue][dashed][line width = 0.8pt][rotate around={240:(0,0)}]
  (2.70,-0.15+0.75) arc(  90:-136:0.75);
%
%
\draw[line width = 1.1pt][->] [rotate around={  0:(0,0)}][violet]
  (1,0) arc(0:45:1);
\draw[line width = 1.1pt][->] [rotate around={ 60:(0,0)}][violet]
  (1,0) arc(0:45:1);
\draw[line width = 0.8pt][->] [rotate around={120:(0,0)}]
  (1,0) arc(0:105:1);
\draw[line width = 0.8pt][->] [rotate around={240:(0,0)}]
  (1,0) arc(0:105:1);
\draw[line width = 1.1pt][->] [rotate around={  0:(0,0)}][violet]
  (2.5,-0.15) -- (2.0,-0.15);
\draw[line width = 1.1pt][->] [rotate around={  0:(0,0)}][violet]
  (1.7,-0.15) -- (1.2,-0.15);
\draw[line width = 0.8pt][->] [rotate around={120:(0,0)}]
  (2.5,-0.15) -- (1.2,-0.15);
\draw[line width = 0.8pt][->] [rotate around={240:(0,0)}]
  (2.5,-0.15) -- (1.2,-0.15);
\draw[line width = 0.8pt][->] [rotate around={  0:(0,0)}]
  (1.1,-0.05) to[out=80,in=0] (-1.1,2.3);
\draw[line width = 0.8pt][->] [rotate around={120:(0,0)}]
  (1.1,-0.05) to[out=80,in=0] (-1.1,2.3);
\draw[line width = 0.8pt][->] [rotate around={240:(0,0)}]
  (1.1,-0.05) to[out=80,in=0] (-1.1,2.3);
\draw[line width = 1.1pt][->] [rotate around={  0:(0,0)}][violet]
  (2.7,0) arc(0:50:2.7);
\draw[line width = 1.1pt][->] [rotate around={ 60:(0,0)}][violet]
  (2.7,0) arc(0:55:2.7);
\draw[line width = 0.8pt][->] [rotate around={120:(0,0)}]
  (2.7,0) arc(0:115:2.7);
\draw[line width = 0.8pt][->] [rotate around={240:(0,0)}]
  (2.7,0) arc(0:115:2.7);
{\tiny
\draw[rotate around={ 60:(0,0)}] ( 1.00,-0.15) node{$v_{a}$};
\draw[rotate around={ 57:(0,0)}] ( 2.70,-0.15) node{$v_{a'}$};
\draw[rotate around={  0:(0,0)}] ( 1.85,-0.15) node{$v_{d}$};
\draw[rotate around={  0:(0,0)}] ( 1.00,-0.15) node{$1$};
\draw[rotate around={120:(0,0)}] ( 1.00,-0.15) node{$2$};
\draw[rotate around={240:(0,0)}] ( 1.00,-0.15) node{$3$};
\draw[rotate around={  0:(0,0)}] ( 2.70,-0.15) node{$4$};
\draw[rotate around={120:(0,0)}] ( 2.70,-0.15) node{$5$};
\draw[rotate around={240:(0,0)}] ( 2.70,-0.15) node{$6$};
%
\draw[rotate around={-35:(0,0)}] ( 0.45, 0.78) node{$a_{\Left}$};
\draw[rotate around={ 20:(0,0)}] ( 0.45, 0.78) node{$a_{\Right}$};
\draw[rotate around={120:(0,0)}] ( 0.45, 0.78) node{$b$};
\draw[rotate around={240:(0,0)}] ( 0.45, 0.78) node{$c$};
\draw[rotate around={  0:(0,0)}] ( 2.25, 0.05) node{$d_{\Left}$};
\draw[rotate around={  0:(0,0)}] ( 1.45, 0.05) node{$d_{\Right}$};
\draw[rotate around={120:(0,0)}] ( 1.85, 0.05) node{$e$};
\draw[rotate around={240:(0,0)}] ( 1.85, 0.05) node{$f$};
\draw[rotate around={  0:(0,0)}] ( 0.85, 1.47) node{$d'$};
\draw[rotate around={120:(0,0)}] ( 0.85, 1.47) node{$e'$};
\draw[rotate around={240:(0,0)}] ( 0.85, 1.47) node{$f'$};
\draw[rotate around={-30:(0,0)}] ( 1.45, 2.51) node{$a'_{\Left}$};
\draw[rotate around={ 30:(0,0)}] ( 1.45, 2.51) node{$a'_{\Right}$};
\draw[rotate around={120:(0,0)}] ( 1.45, 2.51) node{$b'$};
\draw[rotate around={240:(0,0)}] ( 1.45, 2.51) node{$c'$};
}
\end{tikzpicture}
\caption{The bound quiver $(\R(\Q),\R(\I))$ of $\R(A)$}
\label{fig in examp:split}
\end{figure}
and, since
\begin{center}
  $\I = \langle ab, bc, ca, dd', ee', ff',$
  $a'b', b'c', c'a', $

  $e'f, e'c', f'd, f'a', d'e, d'b',$
  $a'eb, b'fc, c'da\rangle$,
\end{center}
we have
\begin{center}
  $\R(\I) = \langle a_{\mathrm{R}}b, bc, ca_{\mathrm{L}}, d_{\mathrm{R}}d', ee', ff',$
  $a'_{\mathrm{R}}b', b'c', c'a'_{\mathrm{L}}, $

  $e'f, e'c', f'd_{\mathrm{L}}, f'a'_{\mathrm{L}}, d'e, d'b',$
  $a'_{\mathrm{R}}eb, b'fc, c'd_{\mathrm{L}}d_{\mathrm{R}}a\rangle$,
\end{center}
see the dashed lines in \Pic \ref{fig in examp:split}.
\end{example}

\subsection{$\R$-endomorphism algebra}

The following result provide a method to compute the $\R$-endomorphism algebra $A_{\R}$ of an $\R$-summed module $M_{\R}$ over an SAG-algebra $A$.

\begin{theorem}\label{thm:main 1}
Let $A$ be an SAG-algebra whose bound quiver is $(\Q,\I)$ and $\R$ be an arbitrary left forbidden arrow index of $\Q_1$.
Then $A_{\R} \cong \R(A)$ $(=\kk\R(\Q)/\R(\I))$ is an SAG-algebra.
\end{theorem}

\begin{proof}
Assume $A_{\R} = \kk\Q_{\R}/\I_{R}$, where $\Q_{\R}=((\Q_{\R})_0, (\Q_{\R})_1, \s_{\R}, \t_{\R})$.
Let $\mathcal{X}$ be the full subcategory of $\modcat(A_{\R})$ generated by $\mathfrak{X}$;
$\Irr_{\mathcal{X}}(X_1, X_2)$ be the set of all irreducible homomorphisms in $\mathcal{X}$ from $X_1$ to $X_2$
($X_1$ and $X_2$ are indecomposable $A$-modules);
$\Basis(X_1, X_2)$ be a basis of $\Irr_{\mathcal{X}}(X_1, X_2)$ as a $\kk$-linear space;
and $\displaystyle \Irr(\mathcal{X}) := \bigcup_{X_1,X_2 \in \ind(\mathcal{X})} \Basis(X_1, X_2)$.

We only prove $\R(\Q)=\Q_{\R}$. Indeed, we will provide a one-to-one correspondence between $\R(\Q)$ and $\Irr(\mathcal{X})$ in this proof, and it admits a one-to-one correspondence between the generators of $\I_{\R}$ and the generators of $\R(\I)$.

First of all, by the definition of $A_{\R} = \End_A(A\oplus \bigoplus_{\alpha\in\R} \alpha A)$,
we have a one-to-one correspondence
\begin{align}\label{formula 1 in main 1}
  (\Q_{\R})_0
  \mathop{\to}\limits^{\spadesuit} \mathfrak{X}:=\{ \e_i A \mid i\in \Q_0 \} \cup \{\alpha A\mid \alpha\in \R\}
  \mathop{\to}\limits^{\clubsuit} \Q_0 \cup \Q_0^{\spl} = \R(\Q)_0
\end{align}
from $(\Q_{\R})_0$ to $\R(\Q)_0$, where
$\spadesuit$ is obtained by $\{\mathrm{id}_X \in \R(A) \mid X \in \mathfrak{X}\}$
is a complete set of primitive orthogonal idempotents of $\R(A)$,
and $\clubsuit$ is obtained by Step 1.

Second, we have a one-to-one correspondence
\begin{center}
  $(\Q_{\R})_1 \mathop{\to}\limits^{\heart} \Irr(\mathcal{X})$.
\end{center}
By Lemma \ref{lemm:decomp}, if $\alpha \in \R$, then $h_{\alpha}: P(\t(\alpha)) \to P(\s(\alpha))$ has a
decomposition $h_{\alpha} = fg$ through $\alpha A$.
By Lemma \ref{lemm:decomp-eA}, $g:P(\t(\alpha)) \to \alpha A$ can not be decomposed through any indecomposable projective module which does not isomorphic to $P(\t(\alpha))$,
and $f: \alpha A \to P(\s(\alpha))$ can not be decomposed through any indecomposable projective module which does not isomorphic to $P(\s(\alpha))$.
By Lemma \ref{lemm:decomp-alphaA}, $g$ and $f$ can not be decomposed through any $\beta A$ ($\not\cong \alpha A$, $\beta\in\Q_1$).
That is, $f$ and $g$ can be seen as two base vectors of
$\Irr_{\mathcal{X}}(\alpha A, P(\s(\alpha)))$ and $\Irr_{\mathcal{X}}(P(\t(\alpha),\alpha A))$,
and then, they are corresponded by two arrows $\mathfrak{v}(\alpha A) \to \t(\alpha)$ and $\s(\alpha) \to \mathfrak{v}(\alpha A)$ in $\R(\Q)_1$
under the correspondence $\heart$, respectively.

On the other hand, if $\alpha \in \Q_1\backslash\R$, then, by Lemma \ref{lemm:decomp-two proj},
we obtain that $h_{\alpha}$ is irreducible in $\mathcal{X}$,
and by Lemma \ref{lemm:decomp-two alphaA}, for arbitrary two left forbidden arrows in $\R$,
each non-zero homomorphism in $\Hom_{\mathcal{X}}(\alpha A, \beta A)$ is not irreducible.
Therefore, we have
\begin{align*}
\R(\Q)_1 \mathop{\to}\limits^{\diam}_{1-1} &  \Irr(\mathcal{X}) \\
  = & \bigcup_{X_1,X_2 \in \ind(\mathcal{X})} \Basis(X_1, X_2) \\
  = & \bigcup_{v\in\Q_0,\alpha\in\R } \Basis(P(v), \alpha A)
      \bigcup\bigcup_{v\in\Q_0,\alpha\in\R } \Basis(\alpha A, P(v)) \\
    & \ \ \ \
      \bigcup\bigcup_{\alpha\in\Q_1\backslash\R} \Basis(P(\t(\alpha)), P(\s(\alpha))),
\end{align*}
where $\Basis(P(v), \alpha A)$, $\Basis(\alpha A, P(v))$, and $\Basis(P(\t(\alpha)), P(\s(\alpha)))$ are described by Lemmas \ref{lemm:decomp-eA}, \ref{lemm:decomp-alphaA}, and \ref{lemm:decomp-two proj}.
Therefore, we obtain a one-to-one correspondence
\begin{align}\label{formula 2 in main 1}
  (\Q_{\R})_1 \mathop{\to}\limits^{1-1} (\R(\Q))_1
\end{align}
by $\heart$ and $\diam$. Moreover, it is easy to see that four correspondences $\spadesuit$, $\clubsuit$, $\heart$, and $\diam$ show that the following two diagrams
\begin{center}
  $\xymatrix{
    (\Q_{\R})_1 \ar[d]_{(\ref{formula 2 in main 1})} \ar[r]^{\s_{\R}}
  & (\Q_{\R})_0 \ar[d]^{(\ref{formula 1 in main 1})}  \\
    \R(\Q)_1 \ar[r]_{\R(\s)}
  & \R(\Q)_0
  }$
\
  $\xymatrix{
    (\Q_{\R})_1 \ar[d]_{(\ref{formula 2 in main 1})} \ar[r]^{\t_{\R}}
  & (\Q_{\R})_0 \ar[d]^{(\ref{formula 1 in main 1})}  \\
    \R(\Q)_1 \ar[r]_{\R(\t)}
  & \R(\Q)_0
  }$
\end{center}
commute. Thus, $\Q_{\R} = \R(\Q)$.

Finally, $\R(\Q)$ to be an SAG-algebra is shown in by Remark \ref{rmk:R(A) is string}.
\end{proof}

Notice that if $A$ is a string algebra, then for some left forbidden arrow index $\R$,
it may be holds that $A_{\R} \cong \R(A)$. For example, the string algebra $A$ given by Example \ref{examp:string},
and $\R$ the left forbidden arrow index given by Example \ref{examp:split},
one can check that $A_{\R} \cong \R(A)$ in this instance.

\section{On representation types of SAG-algebras}

Recall that a finite-dimensional Algebra $A$ is said to be {\defines representation-finite}
(resp. {\defines representation-infinite}) if the set $\ind(\modcat(A))$ of all isoclasses of indecomposable $A$-modules is a finite (resp. infinite) set.

\subsection{Representation types of SAG-algebras and $\R$-endomorphism algebras}

Theorem \ref{thm:BR1987} admits that the following lemma.
\begin{lemma} \label{lemm:string repr}
A string algebra is representation-infinite if and only if its bound quiver has at least one band.
\end{lemma}
It can be shown by Brauer-Thrall Theorem, see for example, \cite[Chapter IV, Section IV.5]{ASS2006}.

\begin{theorem} \label{thm:main 2} 
Let $A=\kk\Q/\I$ be an SAG-algebra. Then $A$ is representation-finite if and only if,
for all left forbidden arrow indices $\R$, the $\R$-endomorphism algebra $A_{\R}$ is representation-finite.
\end{theorem}

\begin{proof}
If, for arbitrary left forbidden arrow index $\R$, $A_{\R}$ is always representation-finite,
then $A$ is representation-finite which can be proved by the trivial case $\R=\varnothing$.

Next, assume that $A$ is representation-finite. If there is a forbidden left arrow index $\R$
such that $A_{\R}$ is representation-infinity, then, by Lemma \ref{lemm:string repr},
the bound quiver $(\Q_{\R}, \I_{\R})$ of $A_{\R}$ contains a band $b$.
By $(\Q_{\R})_1 = (\Q_1\backslash\R)\cup\R^{\spl}$, all arrows on $b$ can be divided to three classes:
\begin{itemize}
  \item[(1)] the arrows lying in $(\Q_1\backslash\R)\cup\R^{\spl}$;
  \item[(2)] the arrows lying in $\R^{\spl}$ which are of the form $\alpha_{\Left}$;
  \item[(3)] the arrows lying in $\R^{\spl}$ which are of the form $\alpha_{\Right}$.
\end{itemize}
If there is an arrow on $b$ which is of the form $\alpha_{\Right}: \mathfrak{v}(\alpha A) \to \t(\alpha)$,
then $\alpha_{\Left}: \s(\alpha) \to \mathfrak{v}(\alpha A)$ is also an arrow on $b$.
Otherwise, since $b$ can be seen as a cycle without relation on $(\Q_{\R},\I_{\R})$,
we have two cases as following:
\begin{itemize}
  \item[(A)] there exists an arrow $\beta$ on $b$ such that $\t(\beta)=\mathfrak{v}(\alpha A)$,
    cf. \Pic \ref{fig in thm:SAG-reprtype 1} (I);
  \item[(B)] there exists an arrow $\beta$ on $b$ such that $\s(\beta)=\mathfrak{v}(\alpha A)$,
    cf. \Pic \ref{fig in thm:SAG-reprtype 1} (II).
\end{itemize}
\begin{figure}[htbp]
\centering
\begin{tikzpicture}
\draw[orange!50][line width =  13pt][orange!50](2,0) arc(0:360:2);
\draw[orange] (-2.2,0) node[left]{$b$};
\draw[->][line width = 0.8pt][rotate around={  0:(0,0)}] (2,0) arc(0:40:2);
\draw    [line width = 0.8pt][rotate around={ 45:(0,0)}] (2,0) arc(0:40:2) [dotted];
\draw    [line width = 0.8pt][rotate around={ 90:(0,0)}] (2,0) arc(0:40:2) [dotted];
\draw    [line width = 0.8pt][rotate around={135:(0,0)}] (2,0) arc(0:40:2) [dotted];
\draw    [line width = 0.8pt][rotate around={180:(0,0)}] (2,0) arc(0:40:2) [dotted];
\draw    [line width = 0.8pt][rotate around={225:(0,0)}] (2,0) arc(0:40:2) [dotted];
\draw    [line width = 0.8pt][rotate around={270:(0,0)}] (2,0) arc(0:40:2) [dotted];
\draw[->][line width = 1.6pt][rotate around={315:(0,0)}] (2,0) arc(0:35:2) [red];
\draw (2,-0.2) node{\tiny$\mathfrak{v}(\alpha A)$};
\draw[rotate around={ 22.5:(0,0)}]
      (2,-0.2) node[right]{\tiny$\alpha_{\Right}$};
\draw[rotate around={-22.5:(0,0)}]
      (2,-0.2) node[right]{\tiny$\beta$};
\draw (2.2,-0.8) node[right]{\tiny$\alpha_{\Left}$};
\draw[line width = 0.8pt][<-] (2.2,-0.4) arc(180:215:2);
\draw (0,-2.5) node{(I)};
\end{tikzpicture}
\ \
\begin{tikzpicture}
\draw[orange!50][line width =  13pt][orange!50](2,0) arc(0:360:2);
\draw[orange] (-2.2,0) node[left]{$b$};
\draw[->][line width = 0.8pt][rotate around={  0:(0,0)}] (2,0) arc(0:40:2);
\draw    [line width = 0.8pt][rotate around={ 45:(0,0)}] (2,0) arc(0:40:2) [dotted];
\draw    [line width = 0.8pt][rotate around={ 90:(0,0)}] (2,0) arc(0:40:2) [dotted];
\draw    [line width = 0.8pt][rotate around={135:(0,0)}] (2,0) arc(0:40:2) [dotted];
\draw    [line width = 0.8pt][rotate around={180:(0,0)}] (2,0) arc(0:40:2) [dotted];
\draw    [line width = 0.8pt][rotate around={225:(0,0)}] (2,0) arc(0:40:2) [dotted];
\draw    [line width = 0.8pt][rotate around={270:(0,0)}] (2,0) arc(0:40:2) [dotted];
\draw[<-][line width = 1.6pt][rotate around={315:(0,0)}] (2,0) arc(0:35:2) [red];
\draw (2,-0.2) node{\tiny$\mathfrak{v}(\alpha A)$};
\draw[rotate around={ 22.5:(0,0)}]
      (2,-0.2) node[right]{\tiny$\alpha_{\Right}$};
\draw[rotate around={-22.5:(0,0)}]
      (2,-0.2) node[right]{\tiny$\beta$};
\draw (2.2,-0.8) node[right]{\tiny$\alpha_{\Left}$};
\draw[line width = 0.8pt][<-] (2.2,-0.4) arc(180:215:2);
\draw (0,-2.5) node{(II)};
\end{tikzpicture}
\caption{}
\label{fig in thm:SAG-reprtype 1}
\end{figure}
In the case (A), we have $\beta \alpha_{\Right} \in \I$ by using the definition of SAG-algebra,
it contradicts with $b$ to be a band.
In the case (B), we have $\alpha_{\Left}\alpha_{\Right} \in \I$.
However, by Remark \ref{rmk:R(A) is string} and Theorem \ref{thm:main 1},
it contradicts with $\alpha_{\Left}\alpha_{\Right} \notin \I_R = \R(\I)$.
\end{proof}

\begin{corollary} \label{coro:main 3 pre}
Let $A=\kk\Q/\I$ be an SAG-algebra such that, for some left forbidden arrow index $\R \subseteq \Q_1$, the $\R$-endomorphism algebra $A_{\R}$ is representation-finite. Then $A$ is representation-finite.
\end{corollary}

\begin{proof}
Assume than $A_{\R}$ is representation-finite. We show that $A$ is representation-finite.
If $A$ is representation-infinite, then by Lemma \ref{lemm:string repr},
the bound quiver $(\Q, \I)$ of $A$ contains a band $b = b_1b_2\cdots b_n$,
here, $b_i \in \Q_1\cup\Q_1^{-1}$, $\s(b_i) = v_i$ ($1\le i\le n$), $\t(b_n)=\s(b_1)$.
By using Theorem \ref{thm:main 1}, $(\Q_{\R}, \I_{\R})$ contains a band which is of the form
\[b^{\spl} = b_1^{\spl}b_2^{\spl}\cdots b_n^{\spl}\]
where
\[b_i^{\spl} = \begin{cases}
 b_i^{\spl}, & \text{if } b_i\in \Q_1; \\
 (b_i)_{\Right}^{-1}(b_i)_{\Left}^{-1}, & \text{if } b_i\in \Q_1^{-1} \text{ and } b_i\in\R; \\
 b_i^{-1}, & \text{if } b_i\in \Q_1^{-1} \text{ and } b_i\notin\R.
\end{cases}\]
It contradicts with Lemma \ref{lemm:string repr} as required.
\end{proof}

\begin{corollary} \label{coro:main 3-1}
Let $A=\kk\Q/\I$ be an SAG-algebra. Then the following statements are equivalent:
\begin{itemize}
  \item[\rm(1)] $A$ is representation-finite;
  \item[\rm(2)] there is a left forbidden arrow index $\R$ such that $A_{\R}$ is representation-finite;
  \item[\rm(3)] for arbitrary left forbidden arrow index $\R$ such that $A_{\R}$ is representation-finite.
\end{itemize}
\end{corollary}

\begin{proof}
The statements (1) and (3) are equivalent by using Theorem \ref{thm:main 2}.
Moreover, we have that (2) admits (1) by using Corollary \ref{coro:main 3 pre},
and it is trivial that (3) admits (1). Then this corollary holds.
\end{proof}

\subsection{Representation types of SAG-algebras and CM-Auslander algebras}

Recall that a Gorenstein-projective  (say G-projective for short) $A$-module $G$ is a module with complete projective resolution, that is, there is an exact sequence
\begin{center}
  $ \cdots \rarrow{p_{-2}} P_{-1} \rarrow{p_{-1}} P_0 \rarrow{p_0} P_1 \rarrow{p_1} P_2 \rarrow{p_2} \cdots$
\end{center}
such that
\begin{itemize}
  \item it is $\Hom_A(-,A)$-exact;
  \item $G\cong \kernel(p_1) = \image(p_0)$ holds.
\end{itemize}
We denote $\Gproj(A)$ the full subcategory of $\modcat(A)$ generated by all G-projective modules over $A$,
and denote $\ind(\Gproj(A))$ the set of all indecomposable G-projective modules over $A$ (up to isomorphism).

In \cite[etc]{Kal2015,CSZ2018}, Kalck and Chen$-$Shen$-$Zhou respectively provide the descriptions of G-projective modules over gentle algebra and monomial algebra.
Then we obtain that SAG-algebras are CM-finite, that is, the number of isoclasses of indecomposable G-projective modules is finite.
Thus, we can compute the Cohen-Macaulay Auslander algebra, defined as
\[A^{\CMA} := \End_A\left(\bigoplus_{G\in\ind(\Gproj(A))}G\right),\]
of $A$ by using results in \cite{Kal2015,CSZ2018},
see for example, \cite[etc]{CL2017,CL2019,LZhang2024CM-Auslander}.

\subsubsection{Perfect forbidden cycles}

A forbidden cycle $\C = c_1\cdots c_l$ ($\s(c_i)=i$, $i=1,2,\ldots,l$)
on string pair $(\Q,\I)$ is said to be {\defines perfect} if it satisfies the following two conditions.
\begin{itemize}
  \item For any arrow $\alpha$ ending with some vertex $t$ on $\C$, we have $\alpha c_t \notin \I$;
  \item For any arrow $\beta$ starting with some vertex $t$ on $\C$, we have $c_{t-1}\beta \notin \I$.
\end{itemize}

Perfect forbidden cycles can be used to describe all non-projective indecomposable Gorenstein-projective modules over SAG-algebra.
In particular, the set $\perR$ of all arrows on all perfect forbidden cycles is a left forbidden arrow index,
and call it a {\defines perfect index}.
The term ``perfect'' originates from ``perfect path'' and ``perfect pair'' which is first introduced by Chen$-$Shen$-$Zhou in \cite{CSZ2018}.
The following result is a direct corollary of \cite[Proposition 5.1]{CSZ2018}

\begin{corollary}[\!\!{\cite[Proposition 5.1]{CSZ2018}}] \label{coro:CSZ2018}
An arrowed module $\alpha A$ over an SAG-algebra $A=\kk\Q/\I$ is a non-projective indecomposable G-projective module if and only if $\alpha \in \perR$.
\end{corollary}

\begin{proof}
Recall that a {\defines perfect pair} on a perfect forbidden cycle $\C=c_0c_1\cdots c_{n-1}$ of length $n$ is defined as a sequence which is of the following form
\begin{center}
  $\wp[t] = (c_{\overline{1+t}}, c_{\overline{2+t}}, \ldots, c_{\overline{n-1+t}}, c_{\overline{1+t}})$,
\end{center}
where, for any $m\in \NN$, $\overline{m}$ defined as $m$ modulo $n$,
see \cite[Definition 3.3]{CSZ2018} or cf. \cite[Definition 3.1]{LZhang2024CM-Auslander}.
Then, by using the definition of SAG-algebra and \cite[Proposition 5.1]{CSZ2018},
we obtain this corollary.
\end{proof}

\begin{theorem} \label{thm:main 4} 
Let $A=\kk\Q/\I$ be an SAG-algebra, $\C_1,\ldots,\C_t$ be perfect forbidden cycles on $(\Q,\I)$.
Then $A_{\perR}$ is isomorphic to the CM-Auslander algebra of $A$.
\end{theorem}

\begin{proof}
By Corollary \ref{coro:CSZ2018}, an indecomposable module is a non-projective indecomposable G-projective module if and only if it is isomorphic to $\alpha A$ with $\alpha\in \perR$.
Thus, we obtain
\[ A_{\perR}
 = \End_A \bigg(A\oplus\bigoplus_{\alpha\in\perR} \alpha A\bigg)
 \cong \End_A\bigg(A\oplus\bigoplus_{
         \begin{smallmatrix}
           G \in \ind(\Gproj(A)) \\
           G \text{ is non-projective }
         \end{smallmatrix}
    } G\bigg) \]
\[
 \cong \End_A\bigg(\bigoplus_{G \in \ind(\Gproj(A))}G\bigg)
 = A^{\CMA}.
\]
\end{proof}

Furthermore, we have the following result.

\begin{corollary} \label{coro:main 3-2}
An SAG-algebra is representation-finite if and only if so is its CM-Auslander algebra.
\end{corollary}

\begin{proof}
Let $A$ be an SAG-algebra. Notice that $\perR$ is a left forbidden arrow index,
then, by Corollary \ref{coro:main 3-1} (1) and (2), we have the representation types of $A$ and $A_{\perR}$ coincide.
By Theorem \ref{thm:main 4}, we have the representation types of $A$ and $A^{\CMA}$ coincide.
\end{proof}

\begin{example} \rm \label{examp:SAG}
Let $A=\kk\T/\J$ is given by the bound quiver $(\T,\J)$, where $\T$ is the quiver of the string algebra given in Example \ref{examp:string} and
\begin{center}
  $\J = \langle ab, bc, ca, dd', ee', ff',$
  $a'b', b'c', c'a', $

  $e'f, e'c', f'd, f'a', d'e, d'b',$
  $a'e, b'f, c'd\rangle$.
\end{center}
see \Pic \ref{fig in examp:SAG}.
\begin{figure}[htbp]
\centering
\begin{tikzpicture}[scale=2]
\draw[   red][dotted][line width = 0.8pt][rotate around={  0:(0,0)}]
  ( 0.51, 0.78) to[out=-60,in=60] ( 0.51,-0.78);
\draw[yellow][dotted][line width = 0.8pt][rotate around={120:(0,0)}]
  ( 0.51, 0.78) to[out=-60,in=60] ( 0.51,-0.78);
\draw[  blue][dotted][line width = 0.8pt][rotate around={240:(0,0)}]
  ( 0.51, 0.78) to[out=-60,in=60] ( 0.51,-0.78);
\draw[   red][dotted][line width = 0.8pt][rotate around={  0:(0,0)}]
  ( 1.50,-0.15) -- ( 1.00, 1.00);
\draw[yellow][dotted][line width = 0.8pt][rotate around={120:(0,0)}]
  ( 1.50,-0.15) -- ( 1.00, 1.00);
\draw[  blue][dotted][line width = 0.8pt][rotate around={240:(0,0)}]
  ( 1.50,-0.15) -- ( 1.00, 1.00);
\draw[   red][dashed][line width = 1.2pt][rotate around={  0:(0,0)}]
  ( 2.15,-0.15) arc (180:225:0.55);
\draw[yellow][dashed][line width = 1.2pt][rotate around={120:(0,0)}]
  ( 2.15,-0.15) arc (180:225:0.55);
\draw[  blue][dashed][line width = 1.2pt][rotate around={240:(0,0)}]
  ( 2.15,-0.15) arc (180:225:0.55);
\draw[   red][dotted][line width = 0.8pt][rotate around={  0:(0,0)}]
  (2.70,-0.15-0.60) arc( -90:  90:0.60);
\draw[   red][dashed][line width = 0.8pt][rotate around={  0:(0,0)}]
  (2.70,-0.15+0.75) arc(  90:-136:0.75);
\draw[   red][dashed][line width = 1.6pt][rotate around={  0:(0,0)}]
  (2.70,-0.15-0.45) arc( -90: 180:0.45);
\draw[yellow][dotted][line width = 0.8pt][rotate around={120:(0,0)}]
  (2.70,-0.15-0.60) arc( -90:  90:0.60);
\draw[yellow][dashed][line width = 0.8pt][rotate around={120:(0,0)}]
  (2.70,-0.15+0.75) arc(  90:-136:0.75);
\draw[yellow][dashed][line width = 1.6pt][rotate around={120:(0,0)}]
  (2.70,-0.15-0.45) arc( -90: 180:0.45);
\draw[  blue][dotted][line width = 0.8pt][rotate around={240:(0,0)}]
  (2.70,-0.15-0.60) arc( -90:  90:0.60);
\draw[  blue][dashed][line width = 0.8pt][rotate around={240:(0,0)}]
  (2.70,-0.15+0.75) arc(  90:-136:0.75);
\draw[  blue][dashed][line width = 1.6pt][rotate around={240:(0,0)}]
  (2.70,-0.15-0.45) arc( -90: 180:0.45);
\draw[line width = 0.8pt][->] [rotate around={  0:(0,0)}]
  (1,0) arc(0:105:1);
\draw[line width = 0.8pt][->] [rotate around={120:(0,0)}]
  (1,0) arc(0:105:1);
\draw[line width = 0.8pt][->] [rotate around={240:(0,0)}]
  (1,0) arc(0:105:1);
\draw[line width = 0.8pt][->] [rotate around={  0:(0,0)}]
  (2.5,-0.15) -- (1.2,-0.15);
\draw[line width = 0.8pt][->] [rotate around={120:(0,0)}]
  (2.5,-0.15) -- (1.2,-0.15);
\draw[line width = 0.8pt][->] [rotate around={240:(0,0)}]
  (2.5,-0.15) -- (1.2,-0.15);
\draw[line width = 0.8pt][->] [rotate around={  0:(0,0)}]
  (1.1,-0.05) to[out=80,in=0] (-1.1,2.3);
\draw[line width = 0.8pt][->] [rotate around={120:(0,0)}]
  (1.1,-0.05) to[out=80,in=0] (-1.1,2.3);
\draw[line width = 0.8pt][->] [rotate around={240:(0,0)}]
  (1.1,-0.05) to[out=80,in=0] (-1.1,2.3);
\draw[line width = 0.8pt][->] [rotate around={  0:(0,0)}]
  (2.7,0) arc(0:115:2.7);
\draw[line width = 0.8pt][->] [rotate around={120:(0,0)}]
  (2.7,0) arc(0:115:2.7);
\draw[line width = 0.8pt][->] [rotate around={240:(0,0)}]
  (2.7,0) arc(0:115:2.7);
{\tiny
\draw[rotate around={  0:(0,0)}] ( 1.00,-0.15) node{$1$};
\draw[rotate around={120:(0,0)}] ( 1.00,-0.15) node{$2$};
\draw[rotate around={240:(0,0)}] ( 1.00,-0.15) node{$3$};
\draw[rotate around={  0:(0,0)}] ( 2.70,-0.15) node{$4$};
\draw[rotate around={120:(0,0)}] ( 2.70,-0.15) node{$5$};
\draw[rotate around={240:(0,0)}] ( 2.70,-0.15) node{$6$};
\draw[rotate around={  0:(0,0)}] ( 0.45, 0.78) node{$a$};
\draw[rotate around={120:(0,0)}] ( 0.45, 0.78) node{$b$};
\draw[rotate around={240:(0,0)}] ( 0.45, 0.78) node{$c$};
\draw[rotate around={  0:(0,0)}] ( 1.85, 0.05) node{$d$};
\draw[rotate around={120:(0,0)}] ( 1.85, 0.05) node{$e$};
\draw[rotate around={240:(0,0)}] ( 1.85, 0.05) node{$f$};
\draw[rotate around={  0:(0,0)}] ( 0.85, 1.47) node{$d'$};
\draw[rotate around={120:(0,0)}] ( 0.85, 1.47) node{$e'$};
\draw[rotate around={240:(0,0)}] ( 0.85, 1.47) node{$f'$};
\draw[rotate around={  0:(0,0)}] ( 1.45, 2.51) node{$a'$};
\draw[rotate around={120:(0,0)}] ( 1.45, 2.51) node{$b'$};
\draw[rotate around={240:(0,0)}] ( 1.45, 2.51) node{$c'$};
}
\end{tikzpicture}
\caption{The bound quiver of the SAG-algebra given in Example \ref{fig in examp:SAG}}
\label{fig in examp:SAG}
(The dashed lines represent the relations in $\I$)
\end{figure}
Then $A$ is an SAG-algebra, and $aA$, $bA$, and $cA$ are both non-projective and G-projective since $abc$ is a perfect forbidden cycle.

(1) Notice that $a'b'c'$ is not a perfect forbidden cycle, then $a'A$, $b'A$, and $c'A$ are not G-projective by using Corollary \ref{coro:CSZ2018} (or \cite[Proposition 5.1]{CSZ2018}).
It follows that $\alpha A$ may be not G-projective.

(2) Now we provide an instance for Theorem \ref{thm:main 4} and Corollary \ref{coro:main 3-2}.
Take $\R=\{a,b,c\} = \perR$, then the bound quiver of $A_{\R}$ is shown in \Pic \ref{fig in examp:SAG CMA} which is isomorphic to the CM-Auslander algebra $A^{\CMA} = \kk\T^{\CMA}/\J^{\CMA}$ of $A$,
\begin{figure}[htbp]
\centering
\begin{tikzpicture}[scale=2]
\draw[   red][dotted][line width = 1.1pt][rotate around={  0:(0,0)}]
  ( 0.8, 0.4) arc(90:270:0.5);
\draw[yellow][dotted][line width = 1.1pt][rotate around={120:(0,0)}]
  ( 0.8, 0.4) arc(90:270:0.5);
\draw[  blue][dotted][line width = 1.1pt][rotate around={240:(0,0)}]
  ( 0.8, 0.4) arc(90:270:0.5);
\draw[   red][dotted][line width = 0.8pt][rotate around={  0:(0,0)}]
  ( 1.50,-0.15) -- ( 1.00, 1.00);
\draw[yellow][dotted][line width = 0.8pt][rotate around={120:(0,0)}]
  ( 1.50,-0.15) -- ( 1.00, 1.00);
\draw[  blue][dotted][line width = 0.8pt][rotate around={240:(0,0)}]
  ( 1.50,-0.15) -- ( 1.00, 1.00);
\draw[   red][dashed][line width = 1.2pt][rotate around={  0:(0,0)}]
  ( 2.15,-0.15) arc (180:225:0.55);
\draw[yellow][dashed][line width = 1.2pt][rotate around={120:(0,0)}]
  ( 2.15,-0.15) arc (180:225:0.55);
\draw[  blue][dashed][line width = 1.2pt][rotate around={240:(0,0)}]
  ( 2.15,-0.15) arc (180:225:0.55);
\draw[   red][dotted][line width = 0.8pt][rotate around={  0:(0,0)}]
  (2.70,-0.15-0.60) arc( -90:  90:0.60);
\draw[   red][dashed][line width = 0.8pt][rotate around={  0:(0,0)}]
  (2.70,-0.15+0.75) arc(  90:-136:0.75);
\draw[   red][dashed][line width = 1.6pt][rotate around={  0:(0,0)}]
  (2.70,-0.15-0.45) arc( -90: 180:0.45);
\draw[yellow][dotted][line width = 0.8pt][rotate around={120:(0,0)}]
  (2.70,-0.15-0.60) arc( -90:  90:0.60);
\draw[yellow][dashed][line width = 0.8pt][rotate around={120:(0,0)}]
  (2.70,-0.15+0.75) arc(  90:-136:0.75);
\draw[yellow][dashed][line width = 1.6pt][rotate around={120:(0,0)}]
  (2.70,-0.15-0.45) arc( -90: 180:0.45);
\draw[  blue][dotted][line width = 0.8pt][rotate around={240:(0,0)}]
  (2.70,-0.15-0.60) arc( -90:  90:0.60);
\draw[  blue][dashed][line width = 0.8pt][rotate around={240:(0,0)}]
  (2.70,-0.15+0.75) arc(  90:-136:0.75);
\draw[  blue][dashed][line width = 1.6pt][rotate around={240:(0,0)}]
  (2.70,-0.15-0.45) arc( -90: 180:0.45);
\draw[line width = 1.2pt][->] [rotate around={  0:(0,0)}][violet]
  (1,0) arc(0:45:1);
\draw[line width = 1.2pt][->] [rotate around={ 60:(0,0)}][violet]
  (1,0) arc(0:45:1);
\draw[line width = 1.2pt][->] [rotate around={120:(0,0)}][violet]
  (1,0) arc(0:45:1);
\draw[line width = 1.2pt][->] [rotate around={180:(0,0)}][violet]
  (1,0) arc(0:45:1);
\draw[line width = 1.2pt][->] [rotate around={240:(0,0)}][violet]
  (1,0) arc(0:45:1);
\draw[line width = 1.2pt][->] [rotate around={300:(0,0)}][violet]
  (1,0) arc(0:45:1);
\draw[line width = 0.8pt][->] [rotate around={  0:(0,0)}]
  (2.5,-0.15) -- (1.2,-0.15);
\draw[line width = 0.8pt][->] [rotate around={120:(0,0)}]
  (2.5,-0.15) -- (1.2,-0.15);
\draw[line width = 0.8pt][->] [rotate around={240:(0,0)}]
  (2.5,-0.15) -- (1.2,-0.15);
\draw[line width = 0.8pt][->] [rotate around={  0:(0,0)}]
  (1.1,-0.05) to[out=80,in=0] (-1.1,2.3);
\draw[line width = 0.8pt][->] [rotate around={120:(0,0)}]
  (1.1,-0.05) to[out=80,in=0] (-1.1,2.3);
\draw[line width = 0.8pt][->] [rotate around={240:(0,0)}]
  (1.1,-0.05) to[out=80,in=0] (-1.1,2.3);
\draw[line width = 0.8pt][->] [rotate around={  0:(0,0)}]
  (2.7,0) arc(0:115:2.7);
\draw[line width = 0.8pt][->] [rotate around={120:(0,0)}]
  (2.7,0) arc(0:115:2.7);
\draw[line width = 0.8pt][->] [rotate around={240:(0,0)}]
  (2.7,0) arc(0:115:2.7);
{\tiny
\draw[rotate around={  0:(0,0)}] ( 1.00,-0.15) node{$1$};
\draw[rotate around={120:(0,0)}] ( 1.00,-0.15) node{$2$};
\draw[rotate around={240:(0,0)}] ( 1.00,-0.15) node{$3$};
\draw[rotate around={  0:(0,0)}] ( 2.70,-0.15) node{$4$};
\draw[rotate around={120:(0,0)}] ( 2.70,-0.15) node{$5$};
\draw[rotate around={240:(0,0)}] ( 2.70,-0.15) node{$6$};
\draw[rotate around={-35:(0,0)}] ( 0.45, 0.78) node{$a_{\Left}$};
\draw[rotate around={ 20:(0,0)}] ( 0.45, 0.78) node{$a_{\Right}$};
\draw[rotate around={ 60:(0,0)}] ( 1.00,-0.15) node{$v_{b}$};
\draw[rotate around={ 85:(0,0)}] ( 0.45, 0.78) node{$b_{\Left}$};
\draw[rotate around={140:(0,0)}] ( 0.45, 0.78) node{$b_{\Right}$};
\draw[rotate around={180:(0,0)}] ( 1.00,-0.15) node{$v_{c}$};
\draw[rotate around={205:(0,0)}] ( 0.45, 0.78) node{$c_{\Left}$};
\draw[rotate around={260:(0,0)}] ( 0.45, 0.78) node{$c_{\Right}$};
\draw[rotate around={300:(0,0)}] ( 1.00,-0.15) node{$v_{a}$};
\draw[rotate around={  0:(0,0)}] ( 1.85, 0.05) node{$d$};
\draw[rotate around={120:(0,0)}] ( 1.85, 0.05) node{$e$};
\draw[rotate around={240:(0,0)}] ( 1.85, 0.05) node{$f$};
\draw[rotate around={  0:(0,0)}] ( 0.85, 1.47) node{$d'$};
\draw[rotate around={120:(0,0)}] ( 0.85, 1.47) node{$e'$};
\draw[rotate around={240:(0,0)}] ( 0.85, 1.47) node{$f'$};
\draw[rotate around={  0:(0,0)}] ( 1.45, 2.51) node{$a'$};
\draw[rotate around={120:(0,0)}] ( 1.45, 2.51) node{$b'$};
\draw[rotate around={240:(0,0)}] ( 1.45, 2.51) node{$c'$};
}
\end{tikzpicture}
\caption{The bound quiver of the CM-Auslander algebra $A^{\CMA}$,
where $A$ is the SAG-algebra in Example \ref{examp:SAG}}
\label{fig in examp:SAG CMA}
(The dashed lines represent the relations in $\I^{\CMA}$)
\end{figure}
where
\begin{center}
  $\J^{\CMA} = \langle a_{\Right}b_{\Left}, b_{\Right}c_{\Left}, c_{\Right}a_{\Left}, dd', ee', ff',$
  $a'b', b'c', c'a', $

  $e'f, e'c', f'd, f'a', d'e, d'b',$
  $a'e, b'f, c'd\rangle$.
\end{center}
Moreover, since the bound quiver $(\T,\J)$ has a band
\begin{center}
  $B = a'd'^{-1}ae^{-1}b'e'^{-1}bf^{-1}c'f'^{-1}cd^{-1}$,
\end{center}
we obtain that $A$ is representation-infinite by using Lemma \ref{lemm:string repr}.
Notice that $B$ corresponds to the band
\begin{center}
  $B^{\spl} = a'd'^{-1}a_{\Left}a_{\Right}e^{-1}b'e'^{-1}b_{\Left}b_{\Right}f^{-1}c'f'^{-1}c_{\Left}c_{\Right}d^{-1}$
\end{center}
on the bound quiver $(\T^{\CMA},\J^{\CMA})$, then $A^{\CMA}$ is also a representation-infinite SGA-algebra.
\end{example}

\vspace{1cm}

\section*{Acknowledgements}
\begin{itemize}
\item[$\rhd$]
Yu-Zhe Liu is supported by the National Natural Science Foundation of China (Grant No. 12401042),
Guizhou Provincial Basic Research Program (Natural Science) (Grant No. ZK[2024]YiBan066)
and Scientific Research Foundation of Guizhou University (Grant Nos. [2022]53, [2022]65, [2023]16).
\vspace{-2mm}

\item[$\rhd$]
Panyue Zhou is supported by the National Natural Science Foundation of China (Grant No. 12371034)
and the Hunan Provincial Natural Science Foundation of China (Grant No. 2023JJ30008).
\end{itemize}
\vspace{1cm}

\hspace{-4mm}\textbf{Data Availability}\hspace{2mm} Data sharing not applicable to this article as no datasets were generated or analysed during
the current study.
\vspace{4mm}

\hspace{-4mm}\textbf{Conflict of Interests}\hspace{2mm} The authors declare that they have no conflicts of interest to this work.

\newpage


  \bibliographystyle{alpha} 

 \bibliography{referLiu20241004}

\end{document}